\documentclass[]{amsart}

\makeatletter
\@addtoreset{equation}{section}

\makeatother

\newtheorem{theorem}{Theorem}[section]
\newtheorem{proposition}[theorem]{Proposition}

\newtheorem{lemma}[theorem]{Lemma}
\theoremstyle{definition}
\newtheorem{definition}[theorem]{Definition}

\theoremstyle{remark}
\newtheorem{remark}[theorem]{Remark}

\DeclareMathOperator{\cn}{cn}

\DeclareMathOperator{\sn}{sn}
\DeclareMathOperator{\am}{am}
\DeclareMathOperator{\sign}{sign}

\usepackage{braket} 
\usepackage{amssymb} 
\usepackage{mathtools}
\usepackage{color}
\usepackage[pdftex]{graphicx}
\usepackage{mathrsfs} 
\usepackage{url}
\usepackage{hyperref}
\usepackage[normalem]{ulem}
\usepackage{xcolor}

\begin{document}

\title[Classification and stability of penalized pinned elasticae]{Classification and stability of penalized pinned elasticae}

\author{Marius M\"uller}
\address[M.~M\"uller]{ Institute of mathematics, University of Augsburg, Universit\"atsstrasse 14, 86159 Augsburg, Germany
}
\email{marius1.mueller@uni-a.de}

\author{Kensuke Yoshizawa}
\address[K.~Yoshizawa]{Faculty of Education, Nagasaki University,
1-14 Bunkyo-machi, Nagasaki, 852-8521, Japan
}
\email{k-yoshizaw@nagasaki-u.ac.jp}

\keywords{Pinned elastica, bending energy, stability, elastic flow}
\subjclass[2020]{
Primary: 53A04, Secondary: 49Q10, 53E40} 

\date{\today}

\begin{abstract}
This paper considers critical points of the length-penalized elastic bending energy among planar curves whose endpoints are fixed. 
We classify all critical points with an explicit parametrization. The classification strongly depends on a special penalization parameter $\hat{\lambda}\simeq 0.70107$.
Stability of all the critical points is also investigated, and again the threshold $\hat{\lambda}$ plays a decisive role.
In addition, our explicit parametrization is applied to compare the energy of critical points, leading to uniqueness of minimal nontrivial critical points.
As an application we obtain eventual embeddedness of elastic flows.
\end{abstract}

\maketitle


\section{Introduction}

For an immersed planar curve $\gamma$, the bending energy (also known as the elastic energy) is defined by 
\begin{align*}
B[\gamma]:=\int_\gamma k^2\;\mathrm{d}s, 
\end{align*}
where $k$ and $s$ respectively denote the signed curvature and the arclength parameter of $\gamma$. 
A critical point of the bending energy under the length-constraint is called \emph{Euler's elastica}, and it is known as a model of an elastic rod (cf.\ \cite{APbook,LLbook,Love}).

In this paper we focus on the so-called \emph{modified (or length-penalized) bending energy} 
\[
\mathcal{E}_\lambda[\gamma]:=B[\gamma] +\lambda L[\gamma],
\]
where $\lambda>0$ is called a \emph{penalization parameter}, and $L[\gamma]=\int_\gamma \mathrm{d}s$, which is an object of interest in the recent literature (see e.g.\ \cite{AGP20,Born,LS_JDG,Maddocks1981,Miura20,MW_24X,Sachkov08,Schrader16,Ydcds} for studies of critical points,\ \cite{DP2014,Dia24,DKS2002,KM24,MPP21,OkabeNovaga2017} for gradient flows, and references therein).

In this paper we consider the critical points of $\mathcal{E}_\lambda$ under the so-called \emph{pinned boundary condition}, which prescribes the endpoints up to zeroth order. 
More precisely, given $\ell>0$, let
\begin{align}\label{eq:defA_ell}
    A_\ell:=\Set{ \gamma \in W^{2,2}_{\rm imm}(0,1;\mathbf{R}^2) | \gamma(0)=(0,0), \ \gamma(1)=(\ell,0)}, 
\end{align} 
where $W^{2,2}_{\rm imm}(0,1;\mathbf{R}^2)$ denotes the set of $W^{2,2}$-immersed curves, i.e., 
\[
W^{2,2}_{\rm imm}(0,1;\mathbf{R}^2):=\Set{\gamma\in W^{2,2}(0,1;\mathbf{R}^2) | \,|\gamma'(x)|\neq0 \ \ \text{for all} \ x\in[0,1] }.
\]
In this paper we call a critical point of $\mathcal{E}_\lambda$ in $A_\ell$ a \emph{penalized pinned elastica} (see Definition~\ref{def:PPE} below).

The role of the modified bending energy depends on the penalization parameter $\lambda$. The larger $\lambda $ is, the more dominant is the shortening role as opposed to the straightening effect.
Taking this property into account, we ask the question 
\begin{center}
    How does the penalization parameter $\lambda$ affect the critical points of $\mathcal{E}_\lambda$?
\end{center}

Inspired by this question, we will obtain various properties of penalized pinned elasticae, such as (i) complete classification, (ii) stability results, (iii) 
energy comparison, and (iv) consequences for the elastic flow. 

We first mention the complete classification of penalized pinned elasticae. 
To this end, we here introduce some functions involving the complete elliptic integrals: 
Let $f:[\frac{1}{\sqrt{2}}, 1)\to \mathbf{R}$ and $g:[\frac{1}{\sqrt{2}}, 1)\to \mathbf{R}$ be the functions defined by
\begin{align}
f(q)&:= (4q^4-5q^2+1)\mathrm{K}(q) + (-8q^4+8q^2-1)\mathrm{E}(q), \label{eq:def-f} 
\\
g(q)&:=8\big( 2\mathrm{E}(q)-\mathrm{K}(q) \big)^2 (2q^2-1), \label{eq:def-g} 
\end{align}
where $\mathrm{K}(q)$ and $\mathrm{E}(q)$ denote the complete elliptic integrals, which we introduce in Appendix~\ref{sect:elliptic_functions}.
Then, $f$ has a unique root $\hat{q}\simeq 0.79257$ (cf.\ Lemma~\ref{lem:property-f}) which is also a local maximum of $g$ with 
\begin{align}\label{eq:def-hat_lambda}
g(\hat{q})=:\hat{\lambda}\simeq 0.70107 
\end{align}
(see Lemma~\ref{lem:property-g}).
Following terminology in Definition~\ref{def:sarc_larc_loop}, we can classify penalized pinned elasticae as follows (see also Figure~\ref{fig:intro1}): 
\begin{theorem}[Classification for penalized pinned elasticae]\label{thm:classification-PPE}
Let $\lambda>0$, $\ell>0$, and  
\begin{align}\label{eq:n-lambda-ell}
    n_{\lambda,\ell}:= \left\lceil \sqrt{\tfrac{\lambda\ell^2}{\hat{\lambda}}} \, \right\rceil.
\end{align}
Suppose that $\gamma \in A_\ell$ is a critical point of $\mathcal{E}_\lambda$ in $A_\ell$. 
Then, either $\gamma$ is a trivial line segment or $\gamma$ is, up to reflection and reparametrization, represented by one of the following
\begin{itemize}
    \item[(i)] $(\lambda,\ell,n)$-shorter arc (denoted by $\gamma_{\rm sarc}^{\lambda,\ell,n}$) for some integer $n \geq n_{\lambda,\ell}$; 
    \item[(ii)] $(\lambda,\ell,n)$-longer arc (denoted by $\gamma_{\rm larc}^{\lambda,\ell,n}$)  for some integer $n \geq n_{\lambda,\ell}$; 
    \item[(iii)] $(\lambda,\ell,n)$-loop (denoted by $\gamma_{\rm loop}^{\lambda,\ell,n}$) for some integer $n \geq 1$. 
\end{itemize}
\end{theorem}

\begin{figure}[htbp] 
    \centering
    \includegraphics[scale=0.22]{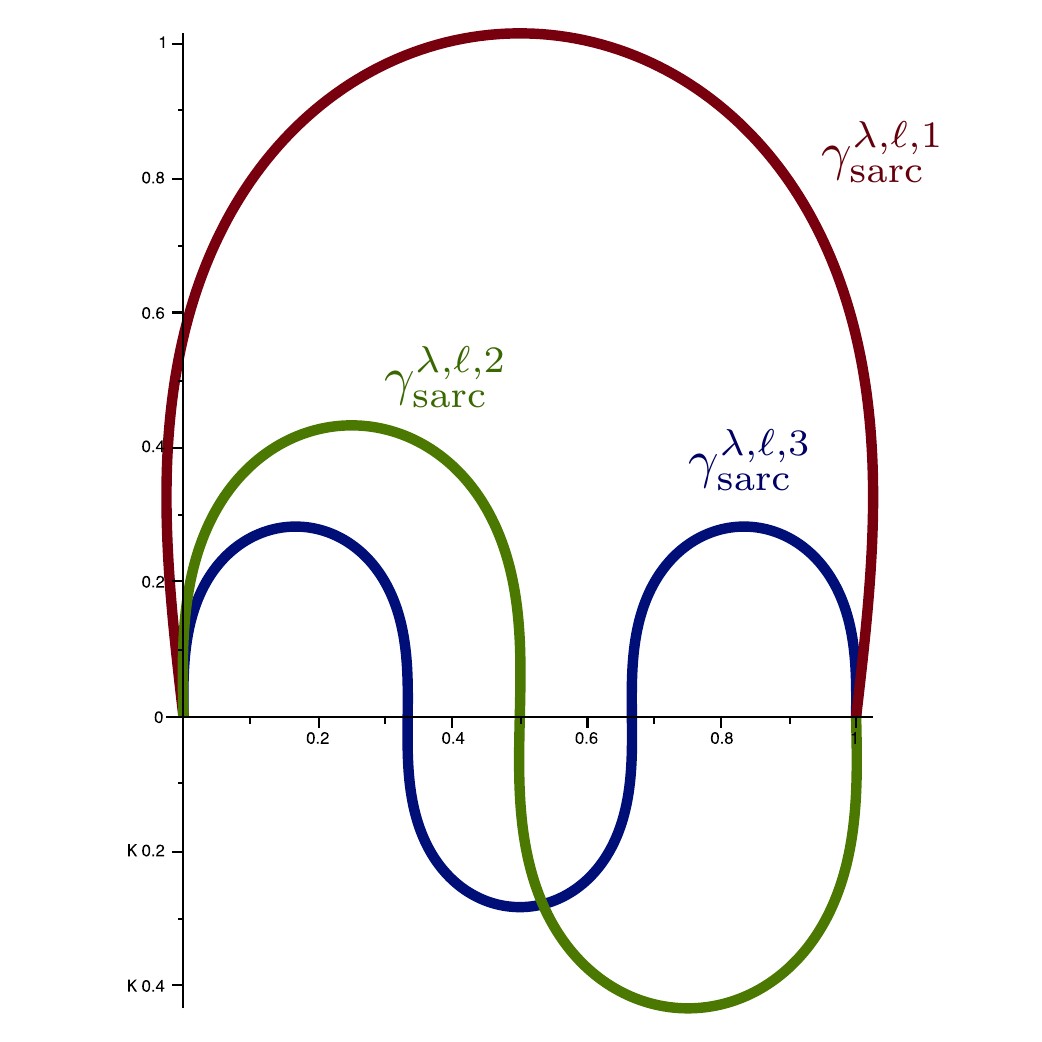}
    \includegraphics[scale=0.22]{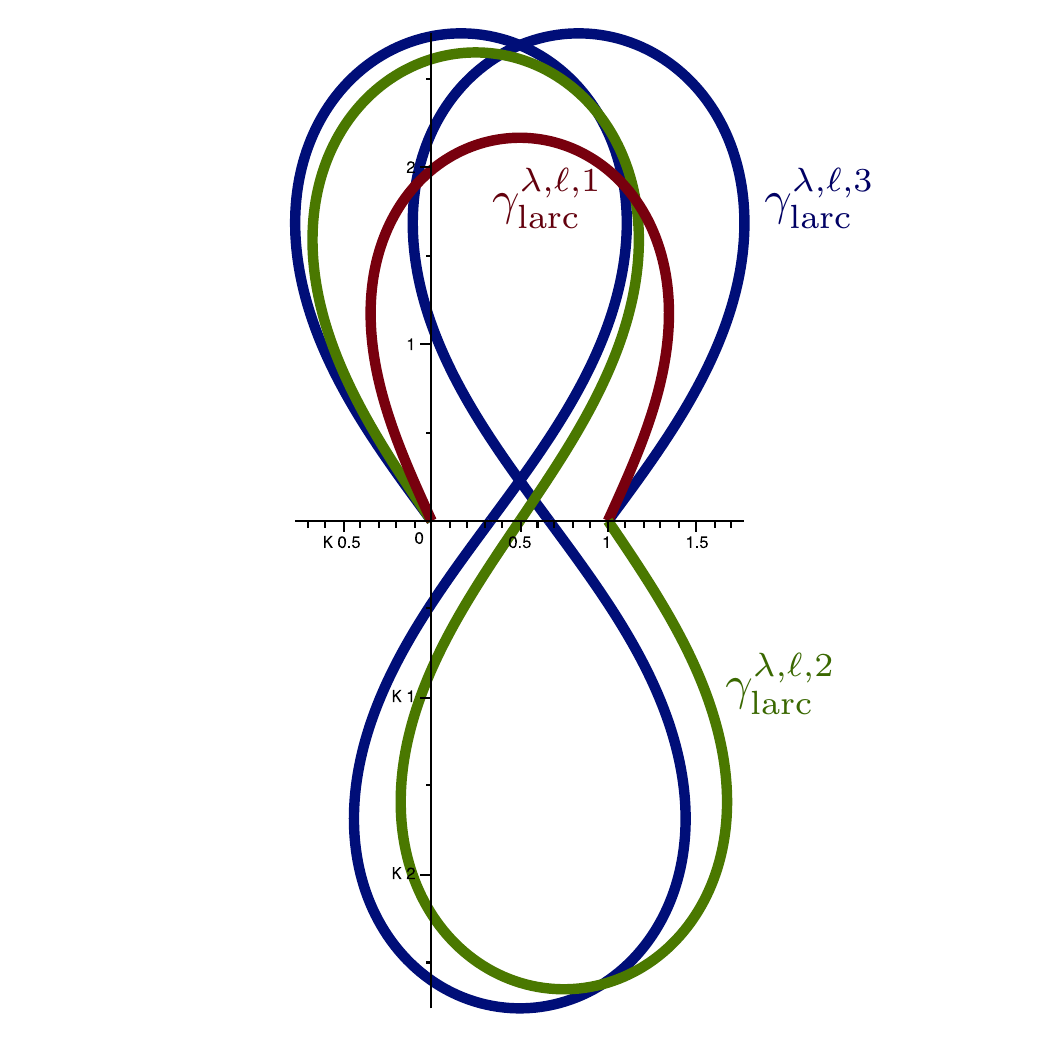}
    \includegraphics[scale=0.22]{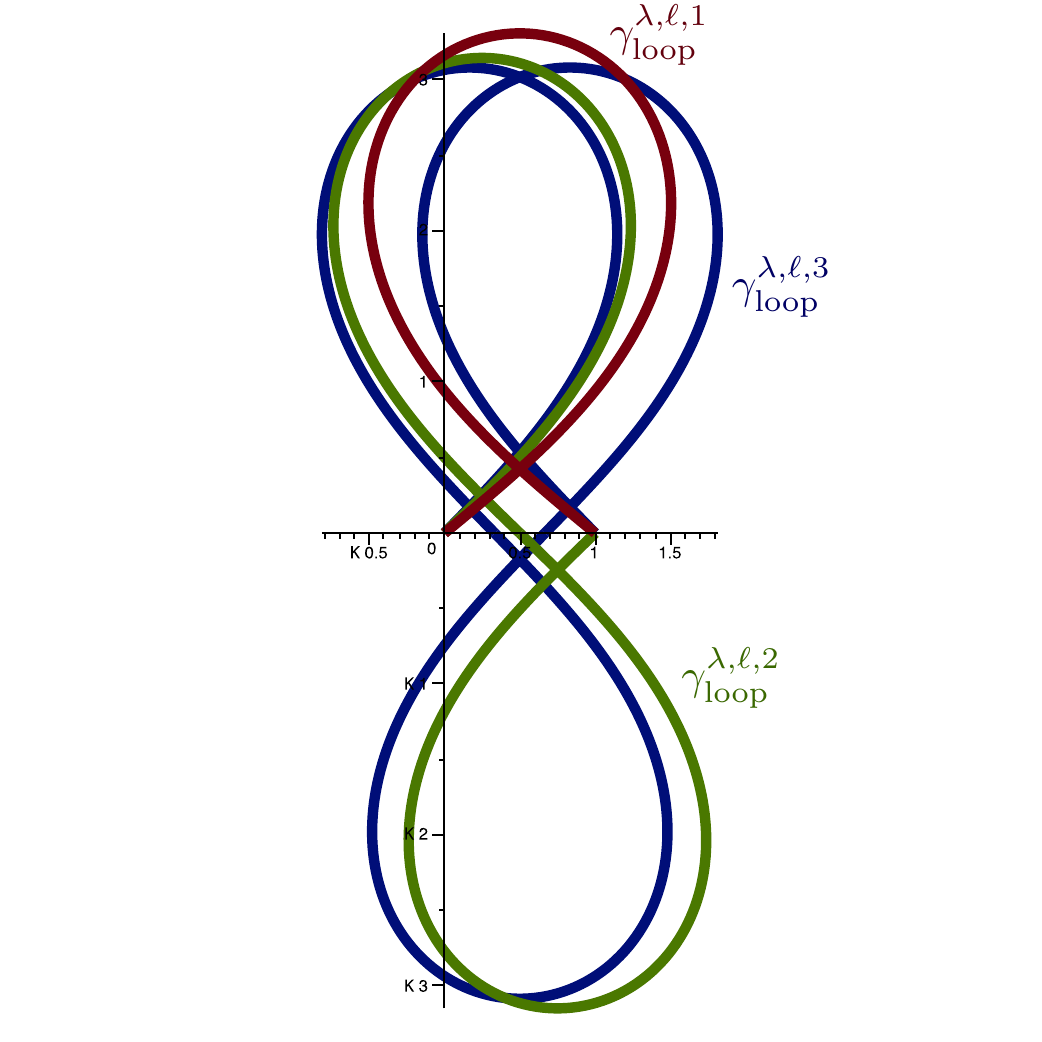}
    \caption{The left $\gamma_{\rm sarc}^{\lambda,\ell,n}$ represents a $(\lambda,\ell,n)$-shorter arc, the middle $\gamma_{\rm larc}^{\lambda,\ell,n}$ a $(\lambda,\ell,n)$-longer arc, the right $\gamma_{\rm loop}^{\lambda,\ell,n}$ a $(\lambda,\ell,n)$-loop, where $\lambda=1/2$ and $\ell=1$.}\label{fig:intro1}
\end{figure}

Explicit formulae for each penalized pinned elasticae are also obtained (see Definition~\ref{def:sarc_larc_loop} below).
Note that the classification strongly depends on the relation of $\lambda \ell^2$ and $\hat{\lambda}$ defined in \eqref{eq:def-hat_lambda}.
More precisely, if $\lambda\ell^2<\hat{\lambda}$, then $n_{\lambda,\ell}=1$ holds so that $\gamma_{\rm sarc}^{\lambda,\ell,1}$ and $\gamma_{\rm larc}^{\lambda,\ell,1}$ appear. 
On the other hand, if $\lambda\ell^2=\hat{\lambda}$, then $\gamma_{\rm sarc}^{\lambda,\ell,1}$ and $\gamma_{\rm larc}^{\lambda,\ell,1}$ coincide, and if  $\lambda\ell^2>\hat{\lambda}$, then $n_{\lambda,\ell}\geq2$ holds so that $\gamma_{\rm sarc}^{\lambda,\ell,1}$ and $\gamma_{\rm larc}^{\lambda,\ell,1}$ do not appear (see Figure~\ref{fig:intro2}). 
Such `bifurcation phenomena' for elasticae have previously been found in e.g.\ \cite[Section 2.3.3-2.3.5]{Lin98TAMS} and \cite[Remark 4.3]{MW_24X}. 

\begin{figure}[htbp]
    \centering
    \includegraphics[scale=0.17]{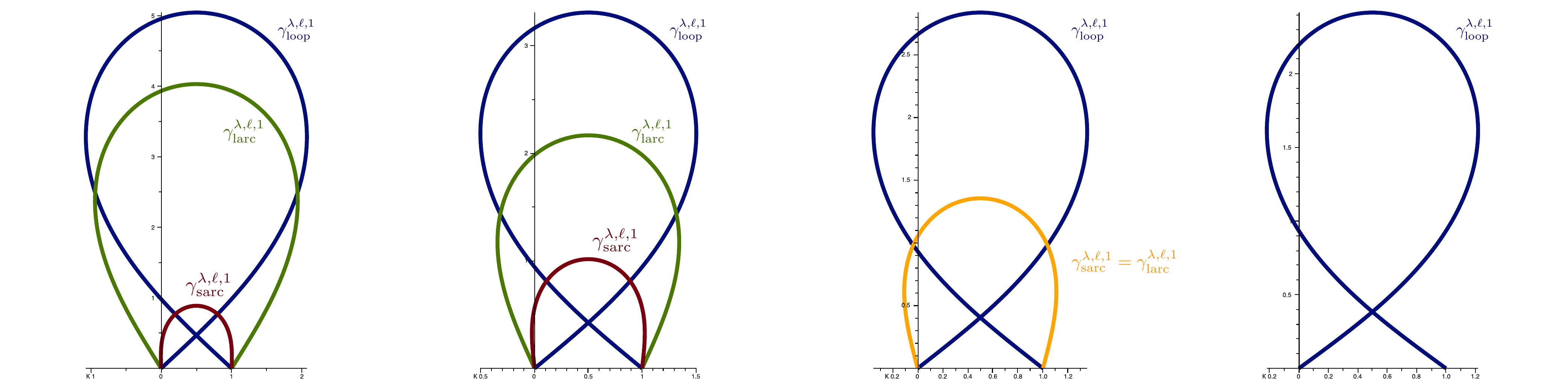}
    \caption{Case $\ell=1$. 
    The penalization parameter $\lambda$ increases from left to right; $\lambda=1/5, 1/2, \hat{\lambda}, 1$.}\label{fig:intro2}
\end{figure}

Next we address the stability (meaning here local minimality) of all the penalized pinned elasticae.
The stability result we obtain again depends on the threshold $\hat{\lambda}$.

\begin{theorem}[Stability of penalized pinned elasticae]\label{thm:stability-PPE}
Let $\gamma \in A_\ell$ be a penalized pinned elastica. 
If $ 0 <\lambda \ell^2 < \hat{\lambda}$, then $\gamma$ is stable if and only if $\gamma$ is either a line segment or $\gamma_{\rm larc}^{\lambda,\ell,1}$.
Moreover, if $\lambda \ell^2 \geq \hat{\lambda}$, then $\gamma$ is stable if and only if $\gamma$ is a line segment.
As a consequence, if $0<\lambda\ell^2<\hat{\lambda}$, then there are exactly two local minimizers of $\mathcal{E}_\lambda$ in $A_\ell$; otherwise the local minimizer is unique.
\end{theorem}

A key tool to prove Theorem~\ref{thm:stability-PPE} is the analysis of the second variation. 
For instance, we show that $\gamma_{\rm sarc}^{\lambda,\ell,1}$ is unstable by finding one perturbation whose second variation takes a negative value (see Lemma~\ref{lem:2nd-derivative}). 
In contrast, in order to show stability \emph{all possible perturbations} need to be taken into account. 
However, it will turn out that in our situation it suffices to show positiveness of the second derivative along \emph{a certain perturbation} of $\gamma_{\rm larc}^{\lambda,\ell,1}$, not all perturbations (see the proof of Theorem~\ref{thm:stability_1-arcs}). 
This substantial simplification is due to a minimizing property of $\gamma_{\rm loop}^{\lambda,\ell,1}$ in a different context from our setting.

We also seek to compare the energies of the above critical points and study which one is energy-minimal (if one excludes the trivial line segment).
As the explicit formulae show, $\mathcal{E}_\lambda[\gamma_{\rm larc}^{\lambda,\ell,1}]$ and $\mathcal{E}_\lambda[\gamma_{\rm loop}^{\lambda,\ell,1}]$ converge to $0$ as $\lambda\to0$. 
This implies that for small $\lambda >0$ it is not easy to determine which has less energy, $\gamma_{\rm larc}^{\lambda,\ell,1}$ or $\gamma_{\rm loop}^{\lambda,\ell,1}$. 
 On the other hand,
for larger $\lambda > 0$ a comparison with $\gamma_{\rm sarc}^{\lambda, \ell,1}$ is also required.
Indeed, since 
\begin{equation}\label{eq:SingleTermComparison}
    B[\gamma_{\rm sarc}^{\lambda,\ell,1}] > B[\gamma_{\rm larc}^{\lambda,\ell,1}] \qquad \textrm{and} \qquad L[\gamma_{\rm sarc}^{\lambda,\ell,1}] < L[\gamma_{\rm larc}^{\lambda,\ell,1}],
\end{equation}
a comparison of $\mathcal{E}_\lambda$ needs to take into account the interaction of both summands. 
 By a rigorous quantitative comparison of the energy, we obtain $\mathcal{E}_\lambda[\gamma_{\rm larc}^{\lambda,\ell,1}]<\mathcal{E}_\lambda[\gamma_{\rm sarc}^{\lambda,\ell,1}]$ and $\mathcal{E}_\lambda[\gamma_{\rm larc}^{\lambda,\ell,1}]<\mathcal{E}_\lambda[\gamma_{\rm loop}^{\lambda,\ell,1}]$ for any $0<\lambda\ell^2<\hat{\lambda}$ (see Lemma~\ref{lem:energy-comparison-sarc_vs_larc} and Lemma~\ref{lem:energy-comparison-arc_vs_loop}, respectively).
In addition, as a consequence of our energy-comparison result, we also obtain uniqueness of minimal nontrivial critical points as follows. 

\begin{theorem}[Uniqueness of nontrivial minimal penalized pinned elasticae]\label{thm:unique_nontrivial_PPE}
Let $\lambda>0$ and $\ell>0$.
\begin{itemize}
    \item[(i)] If $0<\lambda\ell^2 \leq \hat{\lambda}$, then $\gamma_{\rm larc}^{\lambda,\ell,1}$ is a unique minimizer of $\mathcal{E}_\lambda$ among nontrivial penalized pinned elasticae (up to reflection and reparametrization).
    \item[(ii)] If $\lambda\ell^2 > \hat{\lambda}$, then $\gamma_{\rm loop}^{\lambda,\ell,1}$ is a unique minimizer of $\mathcal{E}_\lambda$ among nontrivial penalized pinned elasticae (up to reflection and reparametrization).
\end{itemize} 
\end{theorem}

Our results are applicable to the asymptotic analysis of the $L^2$-gradient flow for $\mathcal{E}_\lambda$, so-called \emph{$\lambda$-elastic flow}, or simply \emph{elastic flow}.
The $\lambda$-elastic flow is defined by an $L^2(ds)$-gradient flow of $\mathcal{E}_\lambda$, and it is given by a one-parameter family of immersed curves $\gamma(x,t):[0,1]\times [0,\infty)\to\mathbf{R}^2$ such that 
\begin{align}\label{eq:elastic_flow}
\partial_t \gamma = -2\nabla^2_s \kappa - |\kappa|^2\kappa + \lambda \kappa,
\end{align}
where $\kappa(x,t):(0,1)\times [0,\infty) \to \mathbf{R}^2$ denotes the curvature vector of $\gamma$, defined by $\kappa=\partial^2_s\gamma$, and $\nabla_s\psi:=\partial_s\psi-(\partial_s \psi, \partial_s \gamma)\partial_s \gamma$ denotes the normal derivative of a smooth vector field $\psi$ along $\gamma$. 
We consider the $\lambda$-elastic flow under the so-called natural (or Navier) boundary condition: 
\begin{align}\label{eq:Navier-BC_flow}
    \gamma(0,t)=(0,0), \ \ \gamma(1,t)=(\ell,0), \ \ \kappa(0,t)=\kappa(1,t)=0 \quad \text{for all} \ \ t\geq0.
\end{align}
Here we are interested in the question of \emph{eventual embeddedness}.  We find a sharp energy threshold below which the above flow destroys any self-intersection in finite time. For more results on (not necessarily eventual) embeddedness of elastic flows we refer to \cite{Bla10,MMR23}.


\begin{theorem}[Eventual embeddedness of elastic flow]\label{thm:embeddedness}
Let $C_{\lambda,\ell}:=\mathcal{E}_\lambda[\gamma_{\rm loop}^{\lambda,\ell,1}]$ and $\gamma_0 \in A_\ell$ be a smooth curve satisfying \eqref{eq:Navier-BC_flow}. 
If
\begin{align} \label{eq:embedded-threshold}
\mathcal{E}_\lambda[\gamma_0] < C_{\lambda,\ell},
\end{align}
then there exists $t_0>0$ such that the $\lambda$-elastic flow with initial datum $\gamma_0$ is embedded for all time $t\geq t_0$. 
\end{theorem}
The value $C_{\lambda,\ell}$ is optimal in the sense that the above conclusion fails if we put any larger constant than $C_{\lambda,\ell}$ in \eqref{eq:embedded-threshold}. Indeed, there exists a smooth curve $\gamma_0$ with $\mathcal{E}_\lambda[\gamma_0] = C_{\lambda,\ell}$ and such that the $\lambda$-elastic flow with initial datum $\gamma_0$ is not embedded for all $t\geq0$. Actually, one may choose $\gamma_0 = \gamma_{\rm loop}^{\lambda,\ell,1}$.

The instability of $\gamma_{\rm loop}^{\lambda,\ell,1}$ in Theorem~\ref{thm:stability-PPE} allows us to construct a nonembedded smooth curve $\gamma_0 \in A_\ell$ satisfying the condition \eqref{eq:embedded-threshold}, cf.\ Remark~\ref{rem:extinction_loop}. 
From Theorem~\ref{thm:embeddedness} we then infer that the flow destroys each self-intersection of $\gamma_0$ in finite time.


We close this introduction by comparing the properties of our penalized pinned elasticae with those already known from \cite{MY_Crelle} about the critical points of $B$ in 
\[
A_{\ell,L}:= \{\gamma \in A_\ell \,|\, L[\gamma]=L\}, 
\]
where $L>\ell$ is given. 
It turns out that the curves of Theorem~\ref{thm:classification-PPE} are also critical points in $A_{\ell,L}$ if $L$ is chosen to be their corresponding length. In \cite{MY_Crelle}, stability of these curves in $A_{\ell,L}$ is investigated. A major difference is that depending on the choice of $L$ it is possible that $\gamma_{\rm sarc}^{\lambda,\ell,1}$ is stable in $A_{\ell,L}$, whereas this is never the case for our penalized problem (cf.\ Theorem~\ref{thm:stability-PPE}). Another difference is that in $A_{\ell,L}$ stable critical points are always unique whereas for our penalized problem the number of stable critical points depends on $\lambda$.
Exposing these differences to the fixed-length problem is a main novelty of the present article.

\subsection*{Organization}
This paper is organized as follows:
In Section~\ref{sect:classification} we give the complete classification of penalized pinned elasticae and prove Theorem~\ref{thm:classification-PPE}. 
In Section~\ref{sect:stability} we investigate the stability (local minimality) of all penalized pinned elasticae to complete the proof of Theorem~\ref{thm:stability-PPE}. 
In Section~\ref{sect:energy_comparison}, we quantitatively compare the energy of penalized pinned elasticae, which yields the proof of Theorem~\ref{thm:unique_nontrivial_PPE}.
In Section~\ref{sect:elastic-flow} we apply our results in Section~\ref{sect:classification}--\ref{sect:energy_comparison} to the analysis of the elastic flow. 

\subsection*{Acknowledgments}
This work was initiated during the second author’s visit at Freiburg University. 
The second author is very thankful to the first author for his warm hospitality and providing a fantastic research atmosphere.
Both authors are grateful to Tatsuya Miura for helpful suggestions.
The authors are grateful to the referee whose suggestions have led to substantial improvements in this article. 
The second author is supported by FMfI Excellent Poster Award 2022 and JSPS KAKENHI Grant Number 24K16951.


\section{Classification}\label{sect:classification}

In this paper we define critical points of $\mathcal{E}_\lambda$ in $A_\ell$ in the following sense: 
\begin{definition}\label{def:PPE}
We call $\gamma$ a \emph{penalized pinned elastica} if $\gamma$ satisfies 
\begin{itemize}
    \item[(i)] $\gamma\in A_\ell$; 
    \item[(ii)] For every smooth family $\{\gamma_\varepsilon\}_{\varepsilon \in (-\varepsilon_0,\varepsilon_0)} \subset A_{\ell}$ with $\varepsilon_0 > 0$ such that $\gamma_0=\gamma$, it follows that 
\begin{equation}\label{eq:first-variation-0}
    \frac{d}{d\varepsilon} \mathcal{E}_\lambda[\gamma_\varepsilon] \Big\vert_{\varepsilon=0} = 0. 
\end{equation}
\end{itemize}
\end{definition}
Note that any local minimizer is a penalized pinned elastica. 
To begin with, we deduce the Euler--Lagrange equation and an additional (natural) boundary condition for penalized pinned elasticae.

\begin{lemma}[The Euler--Lagrange equation and boundary condition] \label{lem:EL-PPE}
Let $\gamma \in A_\ell$ be a penalized pinned elastica. 
Then, the signed curvature $k$ of the arclength reparametrization of $\gamma$ is analytic on $(0,L)$, where $L:=L[\gamma]$, and satisfies 
\begin{gather}
2k'' + k^3 -\lambda k =0, \label{eq:planar-EL}\\
k(0)=k(L)=0. \label{eq:pinned_BC}
\end{gather}
\end{lemma}
\begin{proof}
First we show that $k$ is smooth on $(0,L)$. 
For arbitrary $\varphi\in W^{2,2}(0,L;\mathbf{R}^2)\cap W^{1,2}_0(0,L;\mathbf{R}^2)$ we 
consider $\eta(x):=\varphi(s(x))$ for $x\in[0,1]$, where $s$ denotes the arclength function, i.e., $s(x):=\int_0^x|\gamma'|$. 
Choosing $\gamma_\varepsilon=\gamma+\varepsilon\eta$ in \eqref{eq:first-variation-0},
and using the known formulae of first derivative of $B$ and $L$ (cf.\ \cite[Lemma~A.1]{MY_AMPA}), we deduce that
\[
\frac{d}{d\varepsilon}\mathcal{E}_\lambda[\gamma + \varepsilon\eta]\Big|_{\varepsilon=0}=\int_0^L \Big( 2k\langle \mathbf{n}, \varphi'' \rangle -3 |k|^2 \langle \mathbf{t}, \varphi' \rangle - \lambda k\langle \mathbf{n}, \varphi \rangle \Big) \; \mathrm{d}s,
\]
where $\mathbf{t}$ and $\mathbf{n}$ denote the unit tangent vector and the unit normal vector of $\tilde{\gamma}$, respectively.
Since this identity holds in particular for all $\varphi \in C^\infty_{\rm c}(0,L;\mathbf{R}^2)$, we can deduce from a standard bootstrap argument that $k$ is of class 
$C^\infty(0,L)$. More careful analysis of the equation, which is by now standard, shows also that $k \in C^\infty([0,L])$.

Next we show that $k$ satisfies \eqref{eq:planar-EL} and \eqref{eq:pinned_BC}.
Let $\phi\in W^{2,2}(0,L)\cap W^{1,2}_0(0,L)$ be arbitrary and consider $\eta(x):=\phi(s(x))\mathbf{n}(s(x))$ and again look at $\gamma_\varepsilon:= \gamma + \varepsilon \eta$. 
With the help of integration by parts, we reduce \eqref{eq:first-variation-0} to 
\begin{align*} 
\left[2 k(s)\phi'(s) \right]_{s=0}^{s=L} + \int_0^L \big( 2k'' + k^3 -\lambda k \big) \phi \; \mathrm{d}s=0.
\end{align*}
Considering arbitrary $\phi \in C_{\rm c}^\infty(0,L)$ this implies that $k$ satisfies \eqref{eq:planar-EL}. 
In addition, by choosing $\phi \in C^\infty([0,1])$ with $\phi(0) = \phi(L) = \phi'(0)= 0$ and $\phi'(L)=1$, we deduce from the above relation that $k$ satisfies $k(L)= 0$. 
Similarly we also obtain $k(0) = 0$.
Analyticity of $k$ immediately follows from the fact that $k$ satisfies the polynomial differential equation \eqref{eq:planar-EL}. 
\end{proof}

For later use we exhibit some elementary properties of $f$ and $g$, defined as in \eqref{eq:def-f} and \eqref{eq:def-g}. 
In the following let $q_*\in(0,1)$ denote the unique zero of $q \mapsto 2\mathrm{E}(q)- \mathrm{K}(q)$, cf.\ Lemma~\ref{lem:elliptic_2E-K}.
The proof of Lemma~\ref{lem:property-f} is postponed to Appendix~\ref{sect:proof-fgeh} since it follows from a straightforward calculation.

\begin{lemma}\label{lem:property-f}
Let $f: [\frac{1}{\sqrt{2}}, 1)\to \mathbf{R}$ be the function defined by \eqref{eq:def-f}. 
Then, there exists a unique $\hat{q}\in(\frac{1}{\sqrt{2}}, q_*)$ such that
\begin{align}\label{eq:hat_q-def}
f(\hat{q})=0.
\end{align}
In addition, 
$f>0$ on $[\tfrac{1}{\sqrt{2}}, \hat{q})$ and $f<0$ on $(\hat{q},1)$.
\end{lemma}

The proof of the following lemma is safely omitted since it immediately follows from the definition of $g$ in \eqref{eq:def-g}.

\begin{lemma}\label{lem:property-g}
Let $g:[\frac{1}{\sqrt{2}},1) \to \mathbf{R}$ be the function defined by \eqref{eq:def-g}.
Then,
\begin{align}\label{eq:diff-g}
g'(q)=\frac{16}{q(1-q^2)}\big(2\mathrm{E}(q)-\mathrm{K}(q) \big) f(q)
\end{align}
for $q\in(\frac{1}{\sqrt{2}},1)$.
In addition, the function $g$ has exactly two local extrema at the points $\hat{q}\simeq 0.79257$ 
and $q_*\simeq 0.90891$. 
More precisely, $
g(\hat{q})=:\hat{\lambda}\simeq 0.70107 
$
is a unique local maximum and $g(q_*)=0$ is a local minimum and $g$ is strictly monotone in $(\frac{1}{\sqrt{2}},\hat{q}]$, $[\hat{q},q_*]$ and $[q_*,1)$. 
\end{lemma}

\begin{figure}[htbp]
    \centering
    \includegraphics[scale=0.22]{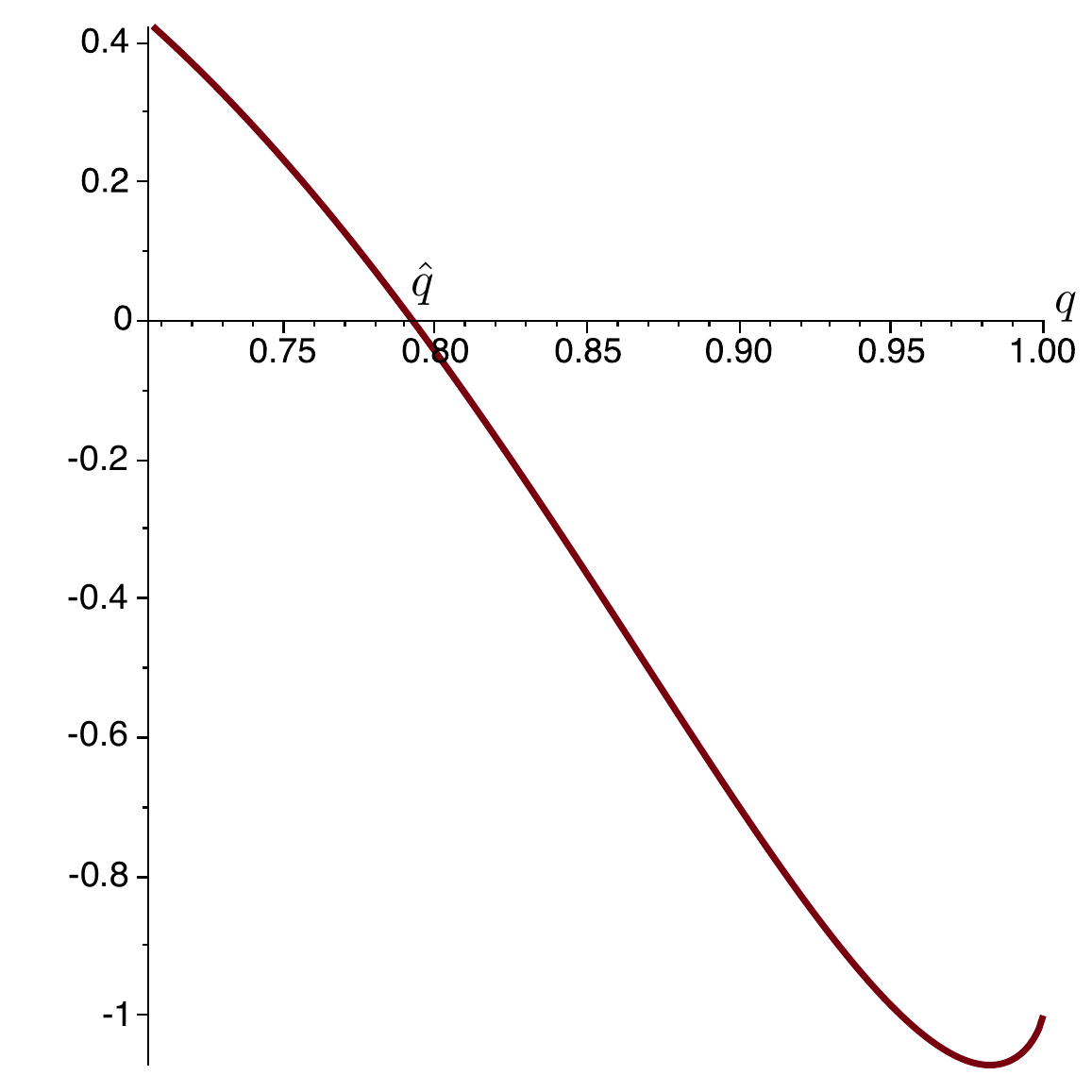}
    \hspace{20pt}
    \includegraphics[scale=0.22]{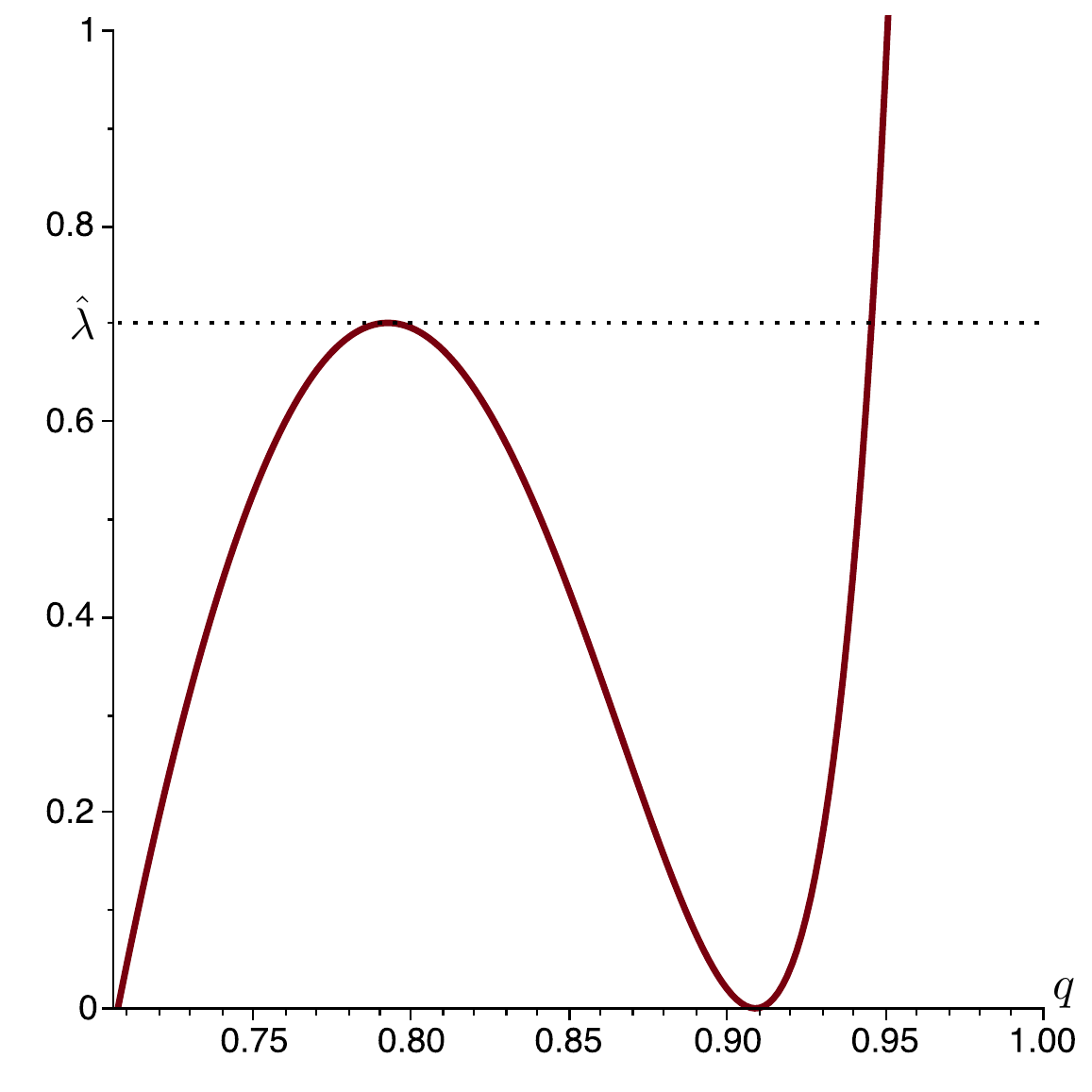}
    \caption{The graph of $f$ (left) and the graph of $g$ (right).}
\end{figure}


Next we analytically estimate $\hat{q}$. 
The proof is postponed to Appendix~\ref{sect:proof-fgeh}. 

\begin{lemma}\label{lem:hat-q-character}
${3}/{5} < \hat{q}^2 < {2}/{3}$.
\end{lemma}

Now we define key moduli for classification of penalized pinned elasticae. 

\begin{definition}\label{def:hat-q_lam}
Let $\hat{\lambda}:=g(\hat{q})$ be the constant given by \eqref{eq:def-hat_lambda}.
\begin{itemize}
    \item [(i)] For $c\in(0,\hat{\lambda}]$, let $q_1(c), q_2(c), q_3(c)$ be the solutions to $g(q)=c$ with 
    \[
    q_1(c)\in(\tfrac{1}{\sqrt{2}},\hat{q}], \quad q_2(c) \in [\hat{q}, q_*), \quad q_3(c)\in(q_*,1).
    \]
    We interpret $q_1(\hat{\lambda})=q_2(\hat{\lambda})=\hat{q}$.
    \item [(ii)] For $c>\hat{\lambda}$, let $q_3(c)\in(q_*,1)$ be a unique solution to $g(q)=c$.
\end{itemize}
\end{definition}

Using these moduli, we prepare terminology for the following critical points.
We will later show that these curves are indeed the only possibilities.

\begin{definition}[Shorter arc, longer arc, and loop]\label{def:sarc_larc_loop}
Let $\lambda>0$, $\ell>0$, and $n\in\mathbf{N}$ be given, and let $n_{\lambda,\ell}\in\mathbf{N}$ as in \eqref{eq:n-lambda-ell}.
\begin{itemize}
    \item[(i)] A curve $\gamma$ is called \emph{$(\lambda,\ell,n)$-shorter arc} if $n\geq n_{\lambda,\ell}$ and if, up to reflection, the arclength parametrization of $\gamma$ is given by 
    \begin{align}\label{def:gamma_sarc-PPE}
        \gamma_{\rm sarc}^{\lambda,\ell,n}(s)&:= \frac{1}{\alpha}
        \begin{pmatrix}
        2\mathrm{E}(\am(\alpha s-\mathrm{K}(q),q),q) + 2\mathrm{E}(q) - \alpha s \\
        2q\cn(\alpha s-\mathrm{K}(q),q)
        \end{pmatrix}, \quad \textrm{where} \\
        q&=q_{1,n}:=q_1\big(\tfrac{\lambda\ell^2}{n^2}), \quad  \alpha=\alpha_{1,n}:=\frac{2n}{\ell}\big(2\mathrm{E}(q_{1,n})-\mathrm{K}(q_{1,n}) \big), \label{eq:q-alpha-sarc}
    \end{align}
    with length $L[\gamma_{\rm sarc}^{\lambda,\ell,n}]=\frac{2n\mathrm{K}(q_{1,n})}{\alpha_{1,n}}$ and  signed curvature 
    \[k(s)=k_{\rm sarc}^{\lambda,\ell,n}(s):=-2\alpha_{1,n} q_{1,n} \cn(\alpha_{1,n} s-\mathrm{K}(q_{1,n}),q_{1,n}).\]
    \item [(ii)] A curve $\gamma$ is called \emph{$(\lambda,\ell,n)$-longer arc} if $n\geq n_{\lambda,\ell}$ and if, up to reflection, the arclength parametrization of $\gamma$ is given by 
    \begin{align}\label{def:gamma_larc-PPE}
        \gamma_{\rm larc}^{\lambda,\ell,n}(s)&:= \frac{1}{\alpha}
        \begin{pmatrix}
        2\mathrm{E}(\am(\alpha s-\mathrm{K}(q),q),q) + 2\mathrm{E}(q) - \alpha s \\
        2q\cn(\alpha s-\mathrm{K}(q),q)
        \end{pmatrix}, \quad \textrm{where} \\
        q&=q_{2,n}:=q_2\big(\tfrac{\lambda\ell^2}{n^2}), \quad \alpha=\alpha_{2,n}:=\frac{2n}{\ell}\big(2\mathrm{E}(q_{2,n})-\mathrm{K}(q_{2,n}) \big), \notag
    \end{align}
    with length $L[\gamma_{\rm larc}^{\lambda,\ell,n}]=\frac{2n\mathrm{K}(q_{2,n})}{\alpha_{2,n}}$ and  signed curvature 
    \[k(s)=k_{\rm larc}^{\lambda,\ell,n}(s)=-2\alpha_{2,n} q_{2,n} \cn(\alpha_{2,n} s-\mathrm{K}(q_{2,n}),q_{2,n}).\]
    \item[(iii)] A curve $\gamma$ is called \emph{$(\lambda,\ell,n)$-loop} if, up to reflection, the arclength parametrization of $\gamma$ is given by 
    \begin{align}\label{def:gamma_loop-PPE}
        \gamma_{\rm loop}^{\lambda,\ell,n}(s)&:= \frac{1}{\alpha}
        \begin{pmatrix}
        -2\mathrm{E}(\am(\alpha s-\mathrm{K}(q),q),q) - 2\mathrm{E}(q) + \alpha s \\
        2q\cn(\alpha s-\mathrm{K}(q),q)
        \end{pmatrix},  \quad \textrm{where}  \\
        q&=q_{3,n}:=q_3\big(\tfrac{\lambda\ell^2}{n^2}), \quad \alpha=\alpha_{3,n}:=\frac{2n}{\ell}\big(\mathrm{K}(q_{3,n}) -2\mathrm{E}(q_{3,n})\big), \notag
        \end{align}
    with length $L[\gamma_{\rm loop}^{\lambda,\ell,n}]=\frac{2n\mathrm{K}(q_{3,n})}{\alpha_{3,n}}$ and  signed curvature 
    \[k(s)=k_{\rm loop}^{\lambda,\ell,n}(s)=2\alpha_{3,n} q_{3,n} \cn(\alpha_{3,n} s-\mathrm{K}(q_{3,n}),q_{3,n}).\]
\end{itemize}
\end{definition}

In what follows we show Theorem~\ref{thm:classification-PPE} to verify that any penalized pinned elastica is either a $(\lambda,\ell,n)$-shorter arc, a $(\lambda,\ell,n)$-longer arc, or a $(\lambda,\ell,n)$-loop.

\begin{proof}[Proof of Theorem~\ref{thm:classification-PPE}]
Let $\gamma \in A_\ell$  be a penalized pinned elastica and $k$ be the signed curvature of $\gamma$.
Recall from Lemma~\ref{lem:EL-PPE} that $k$ satisfies \eqref{eq:planar-EL} and \eqref{eq:pinned_BC}.
Then $k\equiv0$ is a trivial solution to \eqref{eq:planar-EL} and \eqref{eq:pinned_BC}, so in the following we consider the other solutions. 
By \cite[Proposition~3.3]{Lin96}, any solution to the initial value problem for \eqref{eq:planar-EL} with $k(0)=0$ is given by 
\begin{align}\label{eq:kappa-to_be_determined}
k(s)=2\sigma\alpha q\cn(\alpha s+ s_0, q)
\end{align}
for some $\sigma\in\{+,-\}$, $\alpha\geq0$, $s_0\in \mathbf{R}$, and $q\in[0,1)$ such that
\begin{align}\label{eq:lambda-Linner}
    2\alpha^2(2q^2-1)=\lambda. 
\end{align}
Since $\lambda > 0$ we infer that $q > \frac{1}{\sqrt{2}}$. 
By \eqref{eq:pinned_BC} and antiperiodicity of $\cn$ (cf.\ Proposition~\ref{prop:cn}), we may choose $s_0=-\mathrm{K}(q)$ in \eqref{eq:kappa-to_be_determined}.
Denote $L:=L[\gamma]$.
In addition, we deduce from Lemma~\ref{lem:EL-PPE} that $k(L)=\cn(\alpha L-\mathrm{K}(q),q)=0$, which implies that 
\begin{align}\label{eq:length-PPE-n}
    L=\frac{2n}{\alpha}\mathrm{K}(q) \quad \text{for some} \quad n\in\mathbf{N}.
\end{align}
Next we use the boundary condition in the definition of $A_\ell$, cf.\ \eqref{eq:defA_ell}, to gain more information about the parameters.
We see from \eqref{eq:kappa-to_be_determined} that 
\begin{align*}
\int_0^s k(t)\;\mathrm{d}t
&= \int_{-\mathrm{K}(q)}^{\alpha s-\mathrm{K}(q)}  2\sigma q\cn(t,q) \;\mathrm{d}t \\
&=2\sigma\arcsin{\big(q\sn(\alpha s -\mathrm{K}(q), q) \big)} +2\sigma\arcsin{q}, 
\end{align*}
where we used the well known integral formula of $\cn$ (cf.\ \eqref{eq:int-cn}).
This implies that, up to rotation, the tangential angle $\theta$ of $\gamma$  is given by
\begin{align}\label{eq:theta_pm}
\theta(s)= 2\sigma\arcsin{\big(q\sn(\alpha s -\mathrm{K}(q), q) \big)}.
\end{align}
Then, up to rotation and translation, the arclength parametrization $\tilde{\gamma}$ of $\gamma$ is given by 
\begin{align}\label{eq:gamma-before-rotate}
\int_0^s 
\begin{pmatrix}
\cos{\theta(t)} \\
\sin{\theta(t)}
\end{pmatrix}
\mathrm{d}t
=:
\begin{pmatrix}
X_{\sigma,\alpha,q}(s) \\
Y_{\sigma,\alpha,q}(s)
\end{pmatrix}
.
\end{align}
Now we compute 
\begin{align}
\begin{split}\label{eq:Y_+}
    Y_{\sigma,\alpha,q}(s)&=\int_0^s 2 \sin\tfrac{\theta(t)}{2}\cos\tfrac{\theta(t)}{2} \;\mathrm{d}t \\
    &=\int_0^s 2 \sigma q\sn(\alpha t-\mathrm{K}(q),q)\sqrt{1-q^2\sn(\alpha t-\mathrm{K}(q),q)^2} \;\mathrm{d}t \\
    &=\sigma\frac{2}{\alpha}\int_{-\mathrm{K}(q)}^{\alpha s-\mathrm{K}(q)} q\sn(t,q)\sqrt{1-q^2\sn(t,q)^2} \;\mathrm{d}t \\
    &=-\sigma\frac{2q}{\alpha}\cn(\alpha s-\mathrm{K}(q),q),
    \end{split}
\end{align}
where we used \eqref{eq:int-sn(1-q^2sn^2)} in the last equality.
This together with \eqref{eq:length-PPE-n} implies $Y_{\sigma,\alpha,q}(0) = Y_{\sigma,\alpha,q}(L)=0$.
Next we compute
\begin{align}
\begin{split}\label{eq:X_+_compute}
    X_{\sigma,\alpha,q}(s)&=\int_0^s \big( 1-2\sin^2\tfrac{\theta(t)}{2} \big)\;\mathrm{d}t
    =\int_0^s \big( 1-2q^2\sn(\alpha t-\mathrm{K}(q),q)^2\big)\;\mathrm{d}t \\
    &=\frac{1}{\alpha}\int_{-\mathrm{K}(q)}^{\alpha s-\mathrm{K}(q)} \big( 1-2q^2\sn(t,q)^2\big)\;\mathrm{d}t\\
    &=\frac{1}{\alpha}\Big(2\mathrm{E}(\am(\alpha s-\mathrm{K}(q),q),q) +2\mathrm{E}(q) - \alpha s \Big),
\end{split}
\end{align}
where we used \eqref{eq:int-1-sn^2} in the last equality.
Thus we obtain $X_{\sigma,\alpha,q}(0)=0$.
By \eqref{eq:length-PPE-n} and the fact that $\mathrm{E}(\am(m\mathrm{K}(q),q),q)=m\mathrm{E}(q)$ for all $m\in\mathbf{Z}$, we also obtain 
\begin{align}\label{eq:X_+}
    X_{\sigma,\alpha,q}(L)=\frac{2n}{\alpha}(2\mathrm{E}(q)-\mathrm{K}(q)).
\end{align}
Note that $Y_{\sigma,\alpha,q}(L)=0$ follows from \eqref{eq:Y_+}. 
Observe that by Lemma~\ref{lem:elliptic_2E-K} $2\mathrm{E}(q)-\mathrm{K}(q)>0$ if $q\in(0,q_*)$ and $2\mathrm{E}(q)-\mathrm{K}(q)<0$ if $q\in(q_*,1)$. 
Then, combining the boundary condition $\gamma(1)=\tilde{\gamma}(L)=(\ell,0)$ with  \eqref{eq:gamma-before-rotate} and \eqref{eq:X_+}, we obtain
\begin{align*}
\tilde{\gamma}(s)=A 
\begin{pmatrix}
    X_{\sigma,\alpha,q}(s) \\
    Y_{\sigma,\alpha,q}(s)
\end{pmatrix}
, \quad A=
\begin{cases}
     I, \quad &\text{if } q\in(0,q_*) \\
     -I, \quad &\text{if } q\in(q_*,1)
\end{cases},
\end{align*}
and observe that
\begin{align}\label{eq:scaling-a}
\frac{2n}{\alpha}\big|2\mathrm{E}(q)-\mathrm{K}(q)\big| =\ell
\end{align}
is necessary to hold. 
Combining this together with \eqref{eq:lambda-Linner}, we see that
\begin{align}\label{eq:cond-q-penapinned}
\lambda\ell^2 = 8n^2\big(2\mathrm{E}(q)-\mathrm{K}(q) \big)^2(2q^2-1)=n^2g(q),
\end{align}
where $g$ is the function defined by \eqref{eq:def-g}.
Thus $q\in (\frac{1}{\sqrt{2}},1)$ needs to be either 
\[q_1(\tfrac{\lambda\ell^2}{n^2}) \in (\tfrac{1}{\sqrt{2}},\hat{q}), \quad q_2(\tfrac{\lambda\ell^2}{n^2}) \in (\hat{q}, q_*), \quad \text{or} \quad  q_3(\tfrac{\lambda\ell^2}{n^2})\in(q_*,1).
\]
By Lemma~\ref{lem:property-g} $n \geq n_{\lambda, \ell}$ is necessary if $q$ is either $q_1(\frac{\lambda\ell^2}{n^2})$ or $q_2(\frac{\lambda\ell^2}{n^2})$.

In summary, if $\gamma$ is a penalized pinned elastica except for a trivial one, then for some $n \in \mathbf{N}$ $\tilde{\gamma}$ is given by 
\begin{align}
\tilde{\gamma}(s)&=
\begin{pmatrix}
    X_{\sigma,\alpha,q}(s) \\
    Y_{\sigma,\alpha,q}(s)
\end{pmatrix}
 \text{ with } q=q_1(\tfrac{\lambda\ell^2}{n^2}) \ \text{ or }\ q_2(\tfrac{\lambda\ell^2}{n^2}), \quad \text{or} \quad \notag\\
 \tilde{\gamma}(s)&=-
\begin{pmatrix}
    X_{\sigma,\alpha,q}(s) \\
    Y_{\sigma,\alpha,q}(s)
\end{pmatrix}
\text{ with } q=q_3(\tfrac{\lambda\ell^2}{n^2}), \notag
\end{align} 
where $\alpha>0$ is given by \eqref{eq:scaling-a}.
This yields the desired formulae: In \eqref{def:gamma_sarc-PPE}, \eqref{def:gamma_larc-PPE}, and \eqref{def:gamma_loop-PPE}, we chose the sign $\sigma$ so that $\tilde{\gamma}$ lies in the upper-half plane, i.e., 
\begin{align}\label{eq:choice-sigma}
\sigma= - \ \text{ if }\  q=q_1(\tfrac{\lambda}{n^2}) \ \text{ or }\ q_2(\tfrac{\lambda}{n^2}), \quad 
\sigma= + \ \text{ if }\  q=q_3(\tfrac{\lambda}{n^2}). 
\end{align} 
The proof is complete.
\end{proof}

\begin{remark}\label{rem:m/2-fold}
The curves $\gamma_{\rm sarc}^{\lambda,\ell,n}$, $\gamma_{\rm larc}^{\lambda,\ell,n}$, and $\gamma_{\rm loop}^{\lambda,\ell,n}$ obtained in Theorem~\ref{thm:classification-PPE} are $\frac{n}{2}$-fold well-periodic curves in terms of \cite[Definition 2.6]{MY_Crelle}.
This fact will be used when we discuss the instability (see Subsection~\ref{sect:instability_cut-paste}). 
\end{remark}

\subsection{Properties of penalized pinned elasticae}
Below we report on some geometric properties of penalized pinned elasticae, which can be deduced by explicit formulae in Theorem~\ref{thm:classification-PPE}.

\begin{lemma}[Symmetry of penalized pinned elasticae]\label{lem:symmetry-PPE}
Let $\gamma$ be either $\gamma_{\rm sarc}^{\lambda,\ell,1}$, $\gamma_{\rm larc}^{\lambda,\ell,1}$, or $\gamma_{\rm loop}^{\lambda,\ell,1}$ and denote $L:=L[\gamma]$.
Then, $\gamma$ is reflectionally symmetric in the sense that $\gamma=:(X,Y)$ satisfies
\begin{align}\label{eq:symmetry-PPE}
X(s)+X(L-s)=\ell, \quad Y(s)=Y(L-s), \quad \text{for}\ \ s\in [0,L].
\end{align}
In addition, if $\gamma=\gamma_{\rm loop}^{\lambda,\ell,1}$, then $\gamma$ has a self-intersection, i.e., there is $s\in(0,\frac{L}{2})$ such that $\gamma(s)=\gamma(L-s)$.
\end{lemma}
\begin{proof}
In the interest of brevity we only demonstrate the proof of the case of $\gamma=\gamma_{\rm loop}^{\lambda,\ell,1}=:(X, Y)$ since the argument is fairly parallel in the other cases. 
We deduce from \eqref{def:gamma_loop-PPE} that, for $q=q_{3,1}$ and $\alpha=\alpha_{3,1}$, 
\begin{align*}
    Y(L-s)=\frac{2q}{\alpha}\cn(-\alpha s+\mathrm{K}(q),q)=\frac{2q}{\alpha}\cn(\alpha s-\mathrm{K}(q),q)=Y(s)
\end{align*}
for all $s\in[0,L]$, where we used the evenness of $\cn(\cdot,q)$. 
We also apply the oddness of $\am(\cdot,q)$ and $\mathrm{E}(\cdot,q)$ to obtain 
$
    X(L-s)
    =\frac{2}{\alpha}\mathrm{E}(\am(\alpha s -\mathrm{K}(q),q),q) - \frac{2}{\alpha}\mathrm{E}(q) + \frac{2}{\alpha}\mathrm{K}(q) -s. 
$
Since $\alpha=\alpha_{3,1}=\frac{2}{\ell}(\mathrm{K}(q)-2\mathrm{E}(q))$, we have 
\[
X(L-s)+X(s) = -\frac{4}{\alpha}\mathrm{E}(q) + \frac{2}{\alpha}\mathrm{K}(q) = \ell, \quad \textrm{for all } \ s\in[0,L].
\]

It remains to check that $\gamma_{\rm loop}^{\lambda,\ell,1}$ has a self-intersection.
By reflectional symmetry, it suffices to find $s\in (0,\tfrac{L}{2})$ such that $X(s)=\frac{\ell}{2}$.
Here recall from \eqref{eq:theta_pm} and \eqref{eq:choice-sigma} that the tangential angle of $\gamma_{\rm loop}^{\lambda,\ell,1}$ is given by 
\begin{align*}
    \theta_{\rm loop}^{\lambda,\ell,1}(s) := \pi + 2\arcsin{\big(q\sn(\alpha s -\mathrm{K}(q), q) \big)}.
\end{align*}
Combining this with the fact that $L=\frac{2\mathrm{K}(q)}{\alpha}$, we have 
\begin{align*}
    \theta_{\rm loop}^{\lambda,\ell,1}(\tfrac{L}{2})= \pi.
\end{align*}
This implies that $X'(\frac{L}{2})<0$. 
Moreover, since $X(0)=0$ and $X(\tfrac{L}{2})=\frac{\ell}{2}$, we deduce from the intermediate value theorem that there exists $s\in (0,\tfrac{L}{2})$ such that $X(s)=\frac{\ell}{2}$.
The proof is complete.
\end{proof}


\section{Stability}\label{sect:stability}

In this section we address the question of stability of the penalized pinned elasticae found in Theorem \ref{thm:classification-PPE}. 
As mentioned in the introduction, in this paper stability means local minimality of $\mathcal{E}_\lambda$ in $A_\ell$.

\subsection{Stability of one-mode arcs}
Here we focus on the stability of all penalized pinned elasticae with parameter $n = 1.$

The following lemma ensures that under certain conditions it suffices to investigate the sign of the second derivative along \emph{one particular} perturbation.

\begin{lemma}\label{lem:stability-condition}
Let $\gamma\in A_\ell$.
Assume that there exists a perturbation $\{\gamma_q\}\subset A_\ell$ of $\gamma$ such that $\gamma_{q_0}=\gamma$ for some $q_0\in(0,1)$, such that $\frac{d}{dq}\mathcal{E}_\lambda[\gamma_q]|_{q=q_0} =0$, 
\begin{align}
&\frac{d^2}{dq^2}\mathcal{E}_\lambda[\gamma_q] \Big|_{q=q_0} >0, \label{eq:2nd-derivative-general}
\end{align}
and the following properties hold:
\begin{itemize}
    \item[(i)] the map $(0,1)\ni q \mapsto L[\gamma_q] \in (\ell,\infty)$ is continuous and bijective; 
    \item[(ii)] for each $q\in(0,1)$, a curve $\gamma_q$ is a global minimizer of $B$ in 
    \[
    \Set{\gamma\in W^{2,2}_{\rm imm}(0,1; \mathbf{R}^2) | \gamma(0)=(0,0), \ \gamma(1)=(\ell,0), \ L[\gamma]=L[\gamma_q]}.
    \]
\end{itemize}
Then, $\gamma$ is a stable penalized pinned elastica (i.e. a local minimizer of $\mathcal{E}_\lambda$ in $A_\ell$). 
\end{lemma}
\begin{proof}
By \eqref{eq:2nd-derivative-general} (and the fact that the first derivative vanishes), we can find $\epsilon>0$ such that 
\begin{align}\label{eq:partialy_minimize}
    \mathcal{E}_\lambda[\gamma_{q_0}] \leq \mathcal{E}_\lambda[\gamma_{q}] \quad \text{for any} \quad q\in (q_0-\epsilon, q_0+\epsilon). 
\end{align}
We deduce from property (i) that there exists $\delta>0$ such that 
\begin{equation}\label{eq:q-q0smallerthanepsilon}
     |q-q_0| < \epsilon \quad \text{if} \quad \big|L[\gamma_q]-L[\gamma_{q_0}]\big| < \delta. 
\end{equation}
Using the above $\delta>0$, we now fix an arbitrary $\Gamma \in A_\ell$ with $\|\Gamma - \gamma_{q_0}\|_{W^{2,2}} < \delta$. 
This implies in particular that $|L[\Gamma]- L[\gamma_{q_0}]|< \delta$.
Property (i) also implies that there exists $q\in(0,1)$ such that $L[\Gamma]=L[\gamma_q]$. 
From \eqref{eq:q-q0smallerthanepsilon} follows that $q \in (q_0- \epsilon, q_0+ \epsilon)$ 
Moreover, in view of property (ii), we see that $\mathcal{E}_\lambda[\Gamma] \geq \mathcal{E}_\lambda[\gamma_{q}]$, and this together with \eqref{eq:partialy_minimize} yields that 
\[
    \mathcal{E}_\lambda[\gamma_{q_0}] \leq \mathcal{E}_\lambda[\gamma_{q}] \leq \mathcal{E}_\lambda[\Gamma], 
\]
which completes the proof. 
\end{proof}

Thus it suffices to construct a perturbation of $\gamma_{\rm larc}^{\lambda, \ell, 1}$ satisfying the assumption of Lemma~\ref{lem:stability-condition}. 
To this end we introduce a family which consists of wavelike elasticae. 

\begin{definition}\label{def:q-family}
\begin{itemize}
\item[(i)] For $q\in(0,q_*)$, we define $\gamma_w(\cdot,q) \in A_\ell$ to be a curve whose length is $L[\gamma_w(\cdot,q)]=\frac{\ell \mathrm{K}(q)}{2\mathrm{E}(q)-\mathrm{K}(q)}$ and whose signed curvature $k$ is given by 
\begin{align}\label{eq:q-family_curvature}
k(s)=-2\alpha q \cn(\alpha s -\mathrm{K}(q), q), \quad \textrm{where} \quad  \alpha:= \frac{2}{\ell}\big(2\mathrm{E}(q)-\mathrm{K}(q)\big).
\end{align}
\item[(ii)] For $q\in(q_*,1)$, we define $\gamma_w(\cdot,q) \in A_\ell$ to be a curve whose length is $L[\gamma_w(\cdot,q)]=\frac{\ell \mathrm{K}(q)}{\mathrm{K}(q)-2\mathrm{E}(q)}$ and whose signed curvature $k$ is given by 
\begin{align*}
k(s)=2\alpha q \cn(\alpha s -\mathrm{K}(q), q), \quad \textrm{where} \quad \alpha:= \frac{2}{\ell}\big(\mathrm{K}(q)-2\mathrm{E}(q)\big).
\end{align*}
\end{itemize}
\end{definition}

Notice that for each $q$ one has that $\gamma_w(\cdot,q)$ is a wavelike elastica such that $|\gamma_w(0,q)-\gamma_w(1,q)|=\ell$. The proof of this follows the lines of the proof of Theorem~\ref{thm:classification-PPE}.  
Also note that, up to reparametrization, 
\begin{align}\label{fact:arcs_identity}
\gamma_{\rm sarc}^{\lambda,\ell,1}=\gamma_w(\cdot, q_1(\lambda \ell^2)),\ \gamma_{\rm larc}^{\lambda,\ell,1}=\gamma_w(\cdot, q_2(\lambda \ell^2)), \  \text{and}\ \ \gamma_{\rm loop}^{\lambda,\ell,1}=\gamma_w(\cdot, q_3(\lambda \ell^2)).
\end{align}

\begin{remark}\label{rem:smooth_and_loop}
For $q\in(q_*,1)$ the curve $\gamma_w(\cdot,q)$ has a self-intersection (by Lemma~\ref{lem:symmetry-PPE}), and is of class $C^\infty$ by the fact that $\cn(\cdot,q)$ is smooth. 
These facts will be used in the argument for the elastic flow in Section~\ref{sect:elastic-flow}.
\end{remark}

Here we show that $\{\gamma_w(\cdot,q)\}_{q\in (0,q_*)} \subset A_\ell$ satisfies assumptions (i) and (ii) in Lemma~\ref{lem:stability-condition}.

\begin{lemma}\label{lem:minimizing_each_q}
Let $\{\gamma_w(\cdot,q)\}_{q\in(0,q_*)}\subset A_\ell$ be a family defined in Definition~\ref{def:q-family}.
Then, the following properties hold. 
\begin{itemize}
\item[(i)] The map $(0,q_*)\ni q \mapsto L[\gamma_w(\cdot,q)] \in (\ell,\infty)$ is continuous and bijective. 
\item[(ii)] For each $q\in(0,q_*)$ the curve $\gamma_w(\cdot,q)$ is a minimizer of $B$ in 
\[
A_{\ell,q}:=\Set{\gamma\in W^{2,2}_{\rm imm}(0,1; \mathbf{R}^2) | \gamma(0)=(0,0), \ \gamma(1)=(\ell,0), \ L[\gamma]=L[\gamma_w(\cdot,q)]}.
\]
\end{itemize}
\end{lemma}
\begin{proof}
Property (i) follows from the fact that $Q(q)=2\frac{\mathrm{E}(q)}{\mathrm{K}(q)}-1$ is continuous and strictly decreasing (cf.\ \cite[Lemma B.4]{MR23}), and satisfies $Q(0)=1$ and $Q(q_*)=0$.

Next we show property (ii).
The curve $\gamma_w(\cdot,q)$ coincides with $\hat{\gamma}^-_0$ in terms of \cite[Theorem 1.1]{Ydcds} since the modulus  $q\in(0,q_*)$ is uniquely determined by $\frac{\ell}{L}=2\frac{\mathrm{E}(q)}{\mathrm{K}(q)}-1$ and since the signed curvature of $\gamma_w(\cdot,q)$ is given by 
\[ k(s)=-4\frac{q\mathrm{K}(q)}{L} \cn\Big(2\frac{q\mathrm{K}(q)}{L} s -\mathrm{K}(q), q\Big) \]
(see \cite[proof of Theorem 1.1]{Ydcds} for the coincidence of the signed curvature).
Then it follows from \cite[Theorem 1.3]{Ydcds} that $\gamma_w(\cdot,q)=\hat{\gamma}^-_0$ is a minimizer of $B$ in $A_{\ell,q}$. 
\end{proof}


The following lemma ensures the sign of the second derivative of one-mode penalized pinned elasticae along the perturbation of $\{\gamma_w(\cdot,q)\}\subset A_\ell$.

\begin{lemma}\label{lem:2nd-derivative}
Let $\{\gamma_w(\cdot,q)\}_{q\in(0,q_*)}\subset A_\ell$ and $\{\gamma_w(\cdot,q)\}_{q\in(q_*,1)}\subset A_\ell$ be a family defined in Definition~\ref{def:q-family}.
Then 
\[
\frac{d}{dq}\mathcal{E}_\lambda[\gamma_w(\cdot,q)] =0 \quad \text{if}\quad q=q_i(\lambda \ell^2)\ \ (i=1,2,3).
\]
In addition, the following properties hold. 
\begin{itemize}
    \item [(i)] If $0<\lambda\ell^2<\hat{\lambda}$, then 
    \begin{align}\label{eq:sign-2nd_derivative}
    \frac{d^2}{dq^2}\mathcal{E}_\lambda[\gamma_w(\cdot,q)] \Big|_{q=q_1(\lambda \ell^2)} < 0, \quad \frac{d^2}{dq^2}\mathcal{E}_\lambda[\gamma_w(\cdot,q)] \Big|_{q=q_2(\lambda \ell^2)} > 0. 
    \end{align}
    \item [(ii)] If $\lambda\ell^2=\hat{\lambda}$, then 
    \begin{align}\label{eq:sign-2nd_derivative-critical}
    \frac{d^2}{dq^2}\mathcal{E}_\lambda[\gamma_w(\cdot,q)] \Big|_{q=\hat{q}} = 0 \quad \text{and} \quad \frac{d^3}{dq^3}\mathcal{E}_\lambda[\gamma_w(\cdot,q)] \Big|_{q=\hat{q}} > 0.  
    \end{align}
    \item [(iii)] For all $\lambda>0$ and $\ell>0$, 
    \begin{align}\label{eq:sign-2nd_derivative_loop}
    \frac{d^2}{dq^2}\mathcal{E}_\lambda[\gamma_w(\cdot,q)] \Big|_{q=q_3(\lambda \ell^2)} < 0.  
    \end{align}
\end{itemize}
\end{lemma}
\begin{proof}
By \eqref{eq:q-family_curvature} the bending energy of $\gamma_w(\cdot,q)$ is represented by
\begin{align*}
    B[\gamma_w(\cdot,q)]&=\int_0^{\frac{2\mathrm{K}(q)}{\alpha}} 4\alpha^2q^2 \cn(\alpha s-\mathrm{K}(q),q)^2\;\mathrm{d}s
    =4\alpha q^2 \int_{-\mathrm{K}(q)}^{\mathrm{K}(q)} |\cn(s,q)|^2 \; \mathrm{d}s \\
    &=8\alpha q^2 \cdot \frac{q^2 \mathrm{K}(q)-\mathrm{K}(q)+\mathrm{E}(q)}{q^2} \\
    &=\frac{16}{\ell} \big((q^2-1) \mathrm{K}(q)+\mathrm{E}(q) \big)\big(2\mathrm{E}(q)-\mathrm{K}(q)\big). 
\end{align*}
The derivative formulae of elliptic integrals (cf.\ \eqref{eq:diff-elliptic-int} and \eqref{eq:Ydcds_Lem2.6}) give
\begin{align*}
    \frac{d}{dq}B[\gamma_w(\cdot,q)]&= \frac{16}{\ell} \big((q^2-1) \mathrm{K}(q)+\mathrm{E}(q) \big)\Big(2\frac{\mathrm{E}(q)-\mathrm{K}(q)}{q} - \frac{\mathrm{E}(q)-(1-q^2)\mathrm{K}(q)}{q(1-q^2)} \Big) \\
    & \quad +\frac{16}{\ell} q\mathrm{K}(q)\big(2\mathrm{E}(q)-\mathrm{K}(q)\big) \\
    &= \frac{16}{\ell}\left(\frac{1-2q^2}{q}\mathrm{K}(q)^2 + \frac{1-2q^2}{q(1-q^2)}\mathrm{E}(q)^2+\frac{4q^2-2}{q}\mathrm{K}(q)\mathrm{E}(q) \right).
\end{align*}
Recalling that $L[\gamma_w(\cdot,q)]=\frac{\ell \mathrm{K}(q)}{2\mathrm{E}(q)-\mathrm{K}(q)}$, we have
\begin{align*}
    \frac{d}{dq}L[\gamma_w(\cdot,q)]&= \frac{\ell}{(2\mathrm{E}(q)-\mathrm{K}(q))^2}\left(\frac{2}{q}\mathrm{K}(q)^2 + \frac{2}{q(1-q^2)}\mathrm{E}(q)^2-\frac{4}{q}\mathrm{K}(q)\mathrm{E}(q) \right) \\
    &=\frac{16\ell(2q^2-1)}{g(q)}\left(\frac{1}{q}\mathrm{K}(q)^2 + \frac{1}{q(1-q^2)}\mathrm{E}(q)^2-\frac{2}{q}\mathrm{K}(q)\mathrm{E}(q) \right),
\end{align*}
where $g$ is given by \eqref{eq:def-g}.
Thus, setting 
\[
 \mathrm{I}(q):=\frac{1}{q}\mathrm{K}(q)^2 + \frac{1}{q(1-q^2)}\mathrm{E}(q)^2-\frac{2}{q}\mathrm{K}(q)\mathrm{E}(q), 
\]
we have 
\begin{align}\label{eq:mathrm_I}
    \frac{d}{dq}\mathcal{E}_\lambda[\gamma_w(\cdot,q)]=\frac{16}{\ell}(2q^2-1)\Big(-1+\lambda\frac{\ell^2}{g(q)}\Big)\mathrm{I}(q).
\end{align}
Since $\lambda\ell^2=g(q_1(\lambda\ell^2))=g(q_2(\lambda\ell^2))=g(q_3(\lambda\ell^2))$ holds by definition (cf.\ Definition~\ref{def:hat-q_lam}), it follows that $\frac{d}{dq}\mathcal{E}_\lambda[\gamma_w(\cdot,q)]|_{q=q_i(\lambda \ell^2)}=0$ for $i=1,2,3$.

Next we compute the second derivative. 
It follows from \eqref{eq:mathrm_I} that 
\begin{align}\label{eq:2nd_mathrm_I}
\frac{d^2}{dq^2}\mathcal{E}_\lambda[\gamma_w(\cdot,q)]=\frac{16}{\ell}\big((2q^2-1)\mathrm{I}(q)\big)' \Big(-1+\frac{\lambda\ell^2}{g(q)}\Big) - \frac{16}{\ell}(2q^2-1)\mathrm{I}(q)\frac{\lambda\ell^2g'(q)}{g(q)^2}.
\end{align}
Note that the first term in the right-hand side of \eqref{eq:2nd_mathrm_I} vanishes for $q=q_1(\lambda\ell^2)$,  $q_2(\lambda\ell^2)$, or $q_3(\lambda\ell^2)$. 
Note also that $\mathrm{I}(q)>0$ for all $q\in(0,1)$ since \eqref{eq:diff-elliptic-int} yields that $\frac{\mathrm{E}(q)}{1-q^2} > \mathrm{K}(q)$ and therefore 
\begin{align*}
\mathrm{I}(q) > \frac{1}{q}\mathrm{K}(q)^2+\frac{1}{q}\mathrm{E}(q)\mathrm{K}(q)-\frac{2}{q}\mathrm{E}(q)\mathrm{K}(q) =  \frac{1}{q}\mathrm{K}(q)\big(\mathrm{K}(q)-\mathrm{E}(q)\big)\geq0.
\end{align*}
Therefore, we deduce from \eqref{eq:2nd_mathrm_I} that for $i=1,2,3$
\begin{align}\label{eq:2nd_derivative_sign}
\sign \left(\frac{d^2}{dq^2}\mathcal{E}_\lambda[\gamma_w(\cdot,q)]\Big|_{q=q_i(\lambda\ell^2)}\right) =\sign\left( -g'(q_i(\lambda\ell^2)) \right). 
\end{align}
If $0<\lambda\ell^2<\hat{\lambda}$, then combining \eqref{eq:2nd_derivative_sign} with Lemma~\ref{lem:property-g} and the fact that $q_1(\lambda\ell^2)<\hat{q}<q_2(\lambda\ell^2)$, we obtain \eqref{eq:sign-2nd_derivative}.
If $\lambda\ell^2=\hat{\lambda}$ i.e., $q_1(\lambda\ell^2)=q_2(\lambda\ell^2)=\hat{q}$, then \eqref{eq:2nd_derivative_sign} combined with Lemma~\ref{lem:property-g} implies that $\frac{d^2}{dq^2}\mathcal{E}_\lambda[\gamma_w(\cdot,q)] \big|_{q=\hat{q}}= 0$. 
By differentiating \eqref{eq:2nd_mathrm_I} and using the fact that $g(\hat{q})=\lambda\ell^2$ and $g'(\hat{q})=0$, we obtain
\begin{align*}
\frac{d^3}{dq^3}\mathcal{E}_\lambda[\gamma_w(\cdot,q)]\Big|_{q=\hat{q}} =
-\frac{16}{\ell}(2\hat{q}^2-1)\mathrm{I}(\hat{q})\frac{\lambda\ell^2g''(\hat{q})}{g(\hat{q})^2}. 
\end{align*}
Thus in order to show \eqref{eq:sign-2nd_derivative-critical} it suffices to check $g''(\hat{q})<0$, which immediately follows from the derivative formula \eqref{eq:diff-g} of $g$ and the fact that $f(\hat{q})=0$ and $f'(\hat{q})<0$ (cf.\ \eqref{eq:f-diff_estimate}).
Finally, \eqref{eq:sign-2nd_derivative_loop} follows by the combination of \eqref{eq:2nd_derivative_sign} with Lemma~\ref{lem:property-g} and the fact that $q_3(\lambda\ell^2) \in (q_*,1)$.
The proof is complete.
\end{proof}

\begin{theorem}[Stability of one-mode critical points]\label{thm:stability_1-arcs}
Let $\gamma \in A_\ell$ be a penalized pinned elastica and $\tilde{\gamma}$ denote its arclength parametrization. 
\begin{itemize}
    \item[(i)] If $0<\lambda\ell^2<\hat{\lambda}$ and $\tilde{\gamma}=\gamma^{\lambda,\ell,1}_{\rm larc}$, then $\gamma$ is stable. 
    \item[(ii)] If $0<\lambda\ell^2 < \hat{\lambda}$ and $\tilde{\gamma}=\gamma^{\lambda,\ell,1}_{\rm sarc}$, then $\gamma$ is unstable.
    \item[(iii)] If $\lambda\ell^2= \hat{\lambda}$ and $\tilde{\gamma}=\gamma^{\lambda,\ell,1}_{\rm sarc}= \gamma_{\rm larc}^{\lambda,\ell,1}$, then $\gamma$ is unstable.
    \item[(iv)] If $\tilde{\gamma}=\gamma_{\rm loop}^{\lambda,\ell,1}$, then $\gamma$ is unstable.
\end{itemize}
\end{theorem}
\begin{proof}
Assertion (i) follows by the combination of  Lemma~\ref{lem:stability-condition} with Lemmas~ \ref{lem:minimizing_each_q}, \ref{lem:2nd-derivative} and the fact that $\gamma_{\rm larc}^{\lambda,\ell,1}=\gamma_w(\cdot,q_2(\lambda\ell^2))$ (cf.\ \eqref{fact:arcs_identity}).
Assertions (ii), (iii), and (iv) immediately follow from \eqref{eq:sign-2nd_derivative}, \eqref{eq:sign-2nd_derivative-critical}, and \eqref{eq:sign-2nd_derivative_loop}, respectively.
\end{proof}

\subsection{Instability of higher modes ($n\geq2$)}\label{sect:instability_cut-paste}
In this subsection we show that $\gamma_{\rm sarc}^{\lambda,\ell,n}$, $\gamma_{\rm larc}^{\lambda,\ell,n}$, and $\gamma_{\rm loop}^{\lambda,\ell,n}$ are unstable if $n\geq2$. 
Combining this with the previous (in)stability results we are able to prove Theorem~\ref{thm:stability-PPE} in the end of this subsection. 

We apply the general rigidity principles obtained in \cite{MY_Crelle} to deduce the instability of $\gamma_{\rm sarc}^{\lambda,\ell,n}$, $\gamma_{\rm larc}^{\lambda,\ell,n}$, and $\gamma_{\rm loop}^{\lambda,\ell,n}$ for $n\geq2$.
By Lemma~\ref{lem:EL-PPE} the bending energy $B$ satisfies \cite[Hypotheses (H1') and (H2)]{MY_Crelle} with  the choice of $\mathcal{F}=B$ and clearly satisfies \cite[Hypothesis (H3)]{MY_Crelle}. 
This fact together with Remark~\ref{rem:m/2-fold} allows us to apply \cite[Theorems 2.3, 2.7, and 2.8]{MY_Crelle} to the case of $\mathcal{F}=B$ and $\gamma=\gamma_{\rm sarc}^{\lambda,\ell,n}$, $\gamma_{\rm larc}^{\lambda,\ell,n}$, or $\gamma_{\rm loop}^{\lambda,\ell,n}$.

\begin{theorem}[Instability of more than two modes]\label{thm:instability_higher_mode}
Let $\gamma \in A_\ell$ be a penalized pinned elastica. 
If the arclength parametrization of $\gamma$ is represented by either $\gamma_{\rm sarc}^{\lambda,\ell,n}$, $\gamma_{\rm larc}^{\lambda,\ell,n}$, or $\gamma_{\rm loop}^{\lambda,\ell,n}$ for some $n\geq3$, 
then $\gamma$ is not a local minimizer of $\mathcal{E}_\lambda$ in $A_\ell$.
\end{theorem}
\begin{proof}
Let $\gamma$ be either $\gamma_{\rm sarc}^{\lambda,\ell,n}$, $\gamma_{\rm larc}^{\lambda,\ell,n}$, or $\gamma_{\rm loop}^{\lambda,\ell,n}$ for some $n\geq3$ and $L=L[\gamma]$. 
It follows from the formula of the signed curvature obtained in Theorem~\ref{thm:classification-PPE} that $k(0)=k(\frac{L}{n})=k(\frac{2L}{n})=0$, so that $\gamma$ satisfies \cite[Assumption (2.2)]{MY_Crelle}.
Therefore, by \cite[Theorem 2.3 and Remark 4.3]{MY_Crelle} $\gamma$ is not a local minimizer of $B$ in 
\[
A_{\ell,L}=\Set{\gamma\in W^{2,2}_{\rm imm}(0,1;\mathbf{R}^2) | \gamma(0)=(0,0), \ \gamma(1)=(\ell,0), \ L[\gamma]=L}. 
\]
The proof can now be concluded with the following claim.
\begin{align}\label{eq:instability_condition}
\begin{split}
    &\text{If $\gamma$ is not a local minimizer of $B$ in $A_{\ell,L}$}, \\
    &\text{then $\gamma$ is also not a local minimizer of $\mathcal{E}_\lambda$ in $A_\ell$.}
\end{split}
\end{align}
In fact, if $\gamma$ is not a local minimizer of $B$ in $A_{\ell,L}$, then there exists $\{\gamma_j\}_{j\in \mathbf{N}} \subset A_{\ell,L}$ such that $\|\gamma_j-\gamma\|_{W^{2,2}} \to0$ (as $j\to\infty$) and $B[\gamma_j]<B[\gamma]$ for all $j \in \mathbf{N}$. 
Since $A_{\ell,L}\subset A_{\ell}$ and $L[\gamma_j]=L[\gamma]$, the family $\{\gamma_j\}_{j\in \mathbf{N}}$ also satisfies $\{\gamma_j\}_{j\in \mathbf{N}} \subset A_{\ell}$ and $B[\gamma_j]+\lambda L[\gamma_j] < B[\gamma]+\lambda L[\gamma]$. 
This ensures that $\gamma$ is also not a local minimizer of $\mathcal{E}_\lambda$ in $A_\ell$.
\end{proof}

\begin{theorem}[Instability of two modes]\label{thm:instability_2_mode}
Let $\gamma \in A_\ell$ be a penalized pinned elastica. 
If the arclength parametrization of $\gamma$ is represented by either $\gamma_{\rm sarc}^{\lambda,\ell,2}$, $\gamma_{\rm larc}^{\lambda,\ell,2}$, or $\gamma_{\rm loop}^{\lambda,\ell,2}$, 
then $\gamma$ is not a local minimizer of $\mathcal{E}_\lambda$ in $A_\ell$.
\end{theorem}
\begin{proof}
Since $\gamma_{\rm sarc}^{\lambda,\ell,2}$, $\gamma_{\rm larc}^{\lambda,\ell,2}$, and $\gamma_{\rm loop}^{\lambda,\ell,2}$ are $1$-fold well-periodic curves in the sense of \cite[Definition 2.6]{MY_Crelle} (recall Remark~\ref{rem:m/2-fold}), 
we deduce from \cite[Theorem 2.7]{MY_Crelle} that $\gamma$ is not a local minimizer of $B$ in $A_{\ell,L}$. 
This fact together with \eqref{eq:instability_condition} yields the desired conclusion.
\end{proof}

The proof of Theorem~\ref{thm:stability-PPE} is now already complete.

\begin{proof}[Proof of Theorem~\ref{thm:stability-PPE}]
This is a direct consequence of Theorems~\ref{thm:stability_1-arcs}, \ref{thm:instability_higher_mode}, and \ref{thm:instability_2_mode}, combined with the classification of penalized pinned elasticae in Theorem~\ref{thm:classification-PPE} and the fact that a line segment is a global minimizer.
\end{proof}

\begin{remark}
While $\gamma_{\rm loop}^{\lambda,\ell,1}$ is unstable as in Theorem~\ref{thm:stability-PPE}, it will be shown in Theorem~\ref{thm:unique_nontrivial_PPE} that $\gamma_{\rm loop}^{\lambda,\ell,1}$ is the (unique) minimizer among penalized pinned elasticae except for a trivial global minimizer. Notice carefully that these results are not contradictory. In fact, minimizing a functional among a subset of its critical points is different from minimizing the functional in general.  Since perturbations of critical points are not necessarily critical points, approaches like stability analyses can not be used to address the question of minimality among critical points.

Next we examine the energy landscape in the neighborhood of $\gamma_{\rm loop}^{\lambda, \ell,1}$. 
We can find not only an energy-decreasing perturbation (as in  Lemma~\ref{lem:2nd-derivative}) but also and an energy-increasing perturbation as follows.
Let $\gamma \in A_\ell$ be the reparametrization of $\gamma_{\rm loop}^{\lambda,\ell,1}$, and $0<a<b<1$ be such that $\gamma(a)=\gamma(b)$, i.e.\ the self-intersection point. 
For $|\varepsilon|<1$, define $\gamma_\varepsilon:[0,1]\to\mathbf{R}^2$ by the constant-speed reparametrization of 
\begin{align*}
\begin{cases}
    (\varepsilon+1)\big(\gamma(x)-\gamma(a)\big)+\gamma(a) \quad & x\in[a,b], \\
    \gamma(x) &\text{otherwise}, 
\end{cases}
\end{align*}
i.e., constructed by dilation of the loop. 
Setting 
\[
G(\varepsilon):= \frac{1}{\varepsilon+1}B[\gamma|_{[a,b]}] + (\varepsilon+1)\lambda L[\gamma|_{[a,b]}], 
\]
we can compute the energy gap as $\mathcal{E}_\lambda[\gamma_\varepsilon]-\mathcal{E}_\lambda[\gamma] =G(\varepsilon) - G(0)$.
Thus the second variation along this perturbation is given by $G''(0)=2B[\gamma|_{[a,b]}]>0$. Notice also that the first variation vanishes as $\gamma_{\rm loop}^{\lambda, \ell,1 }$ is a critical point.
Hence, one can find that $\gamma$ is a local minimizer along this direction.
This result together with the existence of a perturbation satisfying \eqref{eq:sign-2nd_derivative_loop} imply that 
$\gamma_{\rm loop}^{\lambda,\ell,1}$ behaves like a saddle point. 
\end{remark}


\section{Energy comparison}\label{sect:energy_comparison}

Next we quantitatively compare the energy among penalized pinned elasticae and as a consequence deduce uniqueness of minimizers of $\mathcal{E}_\lambda$ among penalized pinned elasticae except for a trivial line segment.

To begin with we compute the energy of each penalized pinned elastica using the formulae in Definition~\ref{def:sarc_larc_loop}. 
\begin{lemma}\label{lem:energy-formula}
Let $\gamma_{\rm sarc}^{\lambda,\ell,n}$, $\gamma_{\rm larc}^{\lambda,\ell,n}$, $\gamma_{\rm loop}^{\lambda,\ell,n}$ be as in Definition~\ref{def:sarc_larc_loop}. Then, the energy of $\gamma_{\rm sarc}^{\lambda,\ell,n}$, $\gamma_{\rm larc}^{\lambda,\ell,n}$, $\gamma_{\rm loop}^{\lambda,\ell,n}$ is given by 
\begin{align}\label{eq:energy-PPE}
\mathcal{E}_{i,n}:=\frac{8n^2}{\ell} \big|2\mathrm{E}(q_{i,n})-\mathrm{K}(q_{i,n}) \big| \big( (4q_{i,n}^2 -3 ) \mathrm{K}(q_{i,n})+2\mathrm{E}(q_{i,n}) \big)
\end{align}
with $i=1$, $i=2$, and $i=3$, respectively, where $q_{i,n}=q_i(\frac{\lambda\ell^2}{n^2})$ as in Definition~\ref{def:sarc_larc_loop}.
\end{lemma}
\begin{proof}
We only demonstrate the case of $\gamma_{\rm sarc}^{\lambda,\ell,n}$ since the other cases can be deduced in the same way.
Let $n\geq n_{\lambda,\ell}$, $\alpha=\alpha_{1,n}$, and $q=q_1(\frac{\lambda\ell^2}{n^2})$.
Since
\begin{align}\label{eq:computationOFB}
\begin{split}
B[\gamma_{\rm sarc}^{\lambda,\ell,n}]
&=\int_{0}^L |k_{\rm sarc}^{\lambda,\ell,n}(s)|^2 \; \mathrm{d}s
=\int_0^{\frac{2n\mathrm{K}(q)}{\alpha}} |2\alpha q\cn(\alpha s-\mathrm{K}(q)), q)|^2 \; \mathrm{d}s \\
&=4\alpha q^2 \int_{-\mathrm{K}(q)}^{(2n-1)\mathrm{K}(q)} |\cn(x, q)|^2 \; \mathrm{d}x 
=8n\alpha q^2 \cdot \frac{q^2 \mathrm{K}(q)-\mathrm{K}(q)+\mathrm{E}(q)}{q^2} \\
&=8n\alpha (q^2 \mathrm{K}(q)-\mathrm{K}(q)+\mathrm{E}(q) ), 
\end{split}
\end{align}
combining this with $L[\gamma_{\rm sarc}^{\lambda,\ell,n}]=2n{\mathrm{K}(q)}/{\alpha}$  we see that
\begin{align*}
&B[\gamma_{\rm sarc}^{\lambda,\ell,n}] + \lambda L[\gamma_{\rm sarc}^{\lambda,\ell,n}] \\
= &\ 8n\alpha (q^2 \mathrm{K}(q)-\mathrm{K}(q)+\mathrm{E}(q) ) + 2n\lambda \frac{\mathrm{K}(q)}{\alpha} \notag\\
= &\ \frac{8}{\ell}n^2 \big|2\mathrm{E}(q)-\mathrm{K}(q) \big|\Big( 2\big( q^2 \mathrm{K}(q)-\mathrm{K}(q)+\mathrm{E}(q) \big) +(2q^2-1)\mathrm{K}(q) \Big),
\end{align*}
where in the last equality we used $\alpha={2n}\ell^{-1}(2\mathrm{E}(q)-\mathrm{K}(q))$ as in \eqref{eq:q-alpha-sarc} and \eqref{eq:lambda-Linner} $\lambda=n^2\ell^{-2}g(q)$.
The proof is complete.
\end{proof}

To compare the energy of each penalized pinned elastica, we prepare some functions to quantitatively characterize the energy as follows.
Let $e:(\frac{1}{\sqrt{2}},1) \setminus\{q_*\} \to \mathbf{R}$ and $h:(\frac{1}{\sqrt{2}},1) \to \mathbf{R}$ be defined by
\begin{align}
e(q)&:=\frac{(4q^2-3)\mathrm{K}(q)+2\mathrm{E}(q)}{(2q^2-1)|2\mathrm{E}(q)-\mathrm{K}(q)|}, \label{eq:def-e} \\
h(q)&:=\frac{1}{\sqrt{2q^2-1}} \big((4q^2 - 3) \mathrm{K}(q) + 2 \mathrm{E}(q) \big). \label{eq:def-h}
\end{align}
With the aid of the functions $e$ and $h$ one can provide two distinct formulae for the energies of penalized pinned elasticae.
By Lemma~\ref{lem:energy-formula} and \eqref{eq:cond-q-penapinned} we see that 
\begin{align}\label{eq:energy-e-represent}
\begin{split}
\mathcal{E}_{i,n}
&= \frac{8}{\ell}\lambda\ell^2\cdot \frac{n^2}{\lambda\ell^2}\big|2\mathrm{E}(q_{i,n})-\mathrm{K}(q_{i,n}) \big| \big( (4q_{i,n}^2 -3 ) \mathrm{K}(q_{i,n})+2\mathrm{E}(q_{i,n}) \big) \\
&= 8\lambda\ell g(q_{i,n})^{-1}\big|2\mathrm{E}(q_{i,n})-\mathrm{K}(q_{i,n}) \big| \big( (4q_{i,n}^2 -3 ) \mathrm{K}(q_{i,n})+2\mathrm{E}(q_{i,n}) \big)\\
&= \lambda\ell e(q_{i,n}),
\end{split}
\end{align}
which will be useful when we investigate how the energy depends on $n\in\mathbf{N}$ (see Lemma~\ref{lem:energy-comparison-mode}). 
On the other hand, since $\lambda \ell^2 = n^2g(q_{i,n})$ implies that $|2\mathrm{E}(q_{i,n})-\mathrm{K}(q_{i,n})|=\frac{1}{2\sqrt{2}}n\sqrt{\lambda}\ell(2q_{i,n}^2-1)^{-\frac{1}{2}}$, we have
\begin{align}\label{eq:energy-larc-part2} 
\begin{split}
\mathcal{E}_{i,n}
&= \frac{8n^2}{\ell} \frac{ \sqrt{\lambda} \ell}{2\sqrt{2}n\sqrt{2q_{i,n}^2 -1}}\big( (4q_{i,n}^2 -3 ) \mathrm{K}(q_{i,n})+2\mathrm{E}(q_{i,n}) \big) \\
&= 2\sqrt{2}n\sqrt{\lambda} h(q_{i,n}),
\end{split}
\end{align}
which will be useful for comparing the energy of $\gamma^{\lambda,\ell,1}_{\rm larc}$ and $\gamma^{\lambda,\ell,1}_{\rm loop}$ (see Lemma~\ref{lem:energy-comparison-sarc_vs_larc}). 

We exhibit some elementary properties of $e$ and $h$ in the following lemmas, whose proofs are 
postponed to Appendix~\ref{sect:proof-fgeh}.

\begin{lemma}\label{lem:property-e}
Let $e:(\frac{1}{\sqrt{2}},1) \setminus\{q_*\} \to \mathbf{R}$ be the function defined by \eqref{eq:def-e}.
Then,  
\begin{align}\label{eq:diff-e}
e'(q)=-\frac{4f(q)\mathrm{K}'(q)}{(2q^2-1)^2(2\mathrm{E}(q)-\mathrm{K}(q))|2\mathrm{E}(q)-\mathrm{K}(q)|}, \quad q\in (\tfrac{1}{\sqrt{2}},1) \setminus\{q_*\},
\end{align}
where $f$ is the function defined by \eqref{eq:def-f}.
In particular, $e$ is
\begin{align} \label{eq:monotonicity-e}
\text{ strictly decreasing in } (\tfrac{1}{\sqrt{2}}, \hat{q}) \cup (q_*,1), \quad 
\text{ strictly increasing in } (\hat{q}, q_*).
\end{align}
\end{lemma}

\begin{lemma}\label{lem:property-h}
Let $h:(\frac{1}{\sqrt{2}},1)  \to \mathbf{R}$ be the function defined by \eqref{eq:def-h}.
Then,  
\begin{align}\label{eq:diff-h}
h'(q) = \frac{-f(q)}{(2q^2-1)^\frac{3}{2} q(1-q^2)},
\end{align}
where $f$ is the function defined by \eqref{eq:def-f}. 
In particular, $h$ is decreasing on $(\frac{1}{\sqrt{2}},\hat{q}]$ and increasing on $[\hat{q},1)$.
\end{lemma}

Recall that we interpret the number $n \in \mathbf{N}$ as a \emph{mode}.
Therefore, as in \cite{Ydcds}, it will be naturally expected that, for instance, the energy $\mathcal{E}_\lambda[\gamma_{\rm sarc}^{\lambda,\ell,n}]$ gets larger if $n$ is larger.
However, this does not directly follow since the moduli $q= q_{i,n}$ depend on the choice of $n$, cf.\ Definition \ref{def:sarc_larc_loop}. 
In fact, combining the formula
\begin{align}\label{eq:diff-phi}
\frac{d}{dq}\Big(8|2\mathrm{E}(q)-\mathrm{K}(q)|\big((4q^2-3)\mathrm{K}(q)+2\mathrm{E}(q) \big) \Big) = \begin{cases} \frac{16}{q(1-q^2)}f(q)\mathrm{K}(q) & q < q_* \\ - \frac{16}{q(1-q^2)}f(q)\mathrm{K}(q) & q > q_* \end{cases}
\end{align}
with the fact that $n\mapsto q_{1,n}$ is increasing (cf.\ Definition~\ref{def:hat-q_lam}), we see that $n\mapsto  \eta(n):= |2\mathrm{E}(q_{1,n})-\mathrm{K}(q_{1,n}) | ( (4q_{1,n}^2 -3 ) \mathrm{K}(q_{1,n})+2\mathrm{E}(q_{1,n}))$ is decreasing with respect to $n\in \mathbf{N}$.
This implies that $\mathcal{E}_\lambda[\gamma_{\rm sarc}^{\lambda,\ell,n}]$ consists of the increasing factor $n^2$ and the decreasing factor $ \eta(n)$ with respect to $n\in\mathbf{N}$. 

Nevertheless we can obtain the following monotonicity.

\begin{lemma}\label{lem:energy-comparison-mode}
Let $\lambda>0$, $\ell>0$, and $n\in \mathbf{N}$.
Let $\gamma_{\rm sarc}^{\lambda,\ell,n}$, $\gamma_{\rm larc}^{\lambda,\ell,n}$, and $\gamma_{\rm loop}^{\lambda,\ell,n}$ be the curves defined in Definition~\ref{def:sarc_larc_loop}.
Then, if $n<m$, the following inequalities hold:
\begin{align}\label{eq:energy-PPE-monotone-n}
\mathcal{E}_\lambda[\gamma_{\rm sarc}^{\lambda,\ell,n}] < \mathcal{E}_\lambda[\gamma_{\rm sarc}^{\lambda,\ell,m}], \quad 
\mathcal{E}_\lambda[\gamma_{\rm larc}^{\lambda,\ell,n}] < \mathcal{E}_\lambda[\gamma_{\rm larc}^{\lambda,\ell,m}], \quad 
\mathcal{E}_\lambda[\gamma_{\rm loop}^{\lambda,\ell,n}] < \mathcal{E}_\lambda[\gamma_{\rm loop}^{\lambda,\ell,m}]. 
\end{align}
\end{lemma}

\begin{proof}
By the property of $g$ (cf.\ Lemma~\ref{lem:property-g}), for each $i\in \{1,2,3\}$, $q_i(\frac{\lambda\ell^2}{n^2})$ satisfies
\begin{align*}
    q_1(\tfrac{\lambda\ell^2}{n^2})>q_1(\tfrac{\lambda\ell^2}{m^2}), \quad 
    q_2(\tfrac{\lambda\ell^2}{n^2})<q_2(\tfrac{\lambda\ell^2}{m^2}), \quad 
    q_3(\tfrac{\lambda\ell^2}{n^2})>q_3(\tfrac{\lambda\ell^2}{m^2}), \quad \text{if}\quad n<m.
\end{align*}
Combining this with \eqref{eq:energy-e-represent} and Lemma~\ref{lem:property-e}, we obtain the desired monotonicity \eqref{eq:energy-PPE-monotone-n}.
\end{proof}


Here we show that $\mathcal{E}_\lambda[\gamma_{\rm larc}^{\lambda,\ell,n}] \leq \mathcal{E}_\lambda[\gamma_{\rm sarc}^{\lambda,\ell,n}]$ holds for all $n\geq n_{\lambda,\ell}$.

\begin{lemma}[Energy comparison of shorter arc and longer arc]\label{lem:energy-comparison-sarc_vs_larc}
Let $\lambda,\ell>0$ and $n\geq n_{\lambda,\ell}$ be an integer.
Then, the curves $\gamma_{\rm sarc}^{\lambda,\ell,n}$ and $\gamma_{\rm larc}^{\lambda,\ell,n}$ defined in Definition~\ref{def:sarc_larc_loop} satisfy
\begin{align}\label{eq:comparison_arc}
\mathcal{E}_\lambda[\gamma_{\rm larc}^{\lambda,\ell,n}] \leq \mathcal{E}_\lambda[\gamma_{\rm sarc}^{\lambda,\ell,n}].
\end{align}
Equality in \eqref{eq:comparison_arc} is attained if and only if $\lambda\ell^2=m^2\hat{\lambda}$ for some $m\in \mathbf{N}$ and $n=n_{\lambda,\ell}$.
\end{lemma}
\begin{proof}
First we show that \eqref{eq:comparison_arc} for the case $\lambda\ell^2\leq\hat{\lambda}$ and $n=n_{\lambda,\ell}=1$, i.e., 
\begin{align}\label{eq:comparison_arc_1stStep}
    \mathcal{E}_\lambda[\gamma_{\rm larc}^{\lambda,\ell,1}] \leq \mathcal{E}_\lambda[\gamma_{\rm sarc}^{\lambda,\ell,1}] \quad \text{for any} \ \ \lambda,\ell>0 \ \ \text{with} \ \ \lambda\ell^2 \leq \hat{\lambda}.
\end{align}
Fix $\ell>0$ and define $\Phi:(0,\ell^{-2}\hat{\lambda}]\to\mathbf{R}$ by
\[
\Phi(\lambda):=\mathcal{E}_\lambda[\gamma_{\rm sarc}^{\lambda,\ell,1}]-\mathcal{E}_\lambda[\gamma_{\rm larc}^{\lambda,\ell,1}].
\]
 Hereafter we prove that $\Phi(\lambda)>0$ for any $\lambda \in(0,\ell^{-2}\hat{\lambda})$ and $\Phi(\ell^{-2}\hat{\lambda})=0$.
Denote $q_1=q_{1,1}(\lambda):= q_1(\lambda \ell^2)$ and $q_2=q_{2,1}(\lambda):= q_2(\lambda \ell^2)$ for short.
By Lemma~\ref{lem:energy-formula}, and the fact that $\lambda=\frac{1}{\ell^2}g(q_1)=\frac{1}{\ell^2}g(q_2)$, we deduce that
\begin{align}\label{eq:diff-Phi_energy}
\begin{split}
    \Phi'(\lambda)=&\,\frac{dq}{d\lambda}\Big|_{q=q_1}\frac{d}{dq}\Big(8|2\mathrm{E}(q)-\mathrm{K}(q)|\big((4q^2-3)\mathrm{K}(q)+2\mathrm{E}(q) \big) \Big)\Big|_{q=q_1} \\
    &- \frac{dq}{d\lambda}\Big|_{q=q_2}\frac{d}{dq}\Big(8|2\mathrm{E}(q)-\mathrm{K}(q)|\big((4q^2-3)\mathrm{K}(q)+2\mathrm{E}(q) \big) \Big)\Big|_{q=q_2} \\
    =&\,\frac{\ell^2}{g'(q_1)}\frac{16}{q_1(1-q_1^2)}f(q_1)\mathrm{K}(q_1)
    -\frac{\ell^2}{g'(q_2)}\frac{16}{q_2(1-q_2^2)}f(q_2)\mathrm{K}(q_2), 
\end{split}
\end{align}
where in the last equality we also used \eqref{eq:diff-phi}. 
Combining this with \eqref{eq:diff-g}, we obtain 
\[
\Phi'(\lambda)=\ell^2\Big(\frac{\mathrm{K}(q_1)}{2\mathrm{E}(q_1)-\mathrm{K}(q_1)} - \frac{\mathrm{K}(q_2)}{2\mathrm{E}(q_2)-\mathrm{K}(q_2)} \Big). 
\]
Since $\mathrm{K}(q)$ and $(2\mathrm{E}(q)-\mathrm{K}(q))^{-1}$ are positive and strictly increasing with respect to $q\in(0,q_*)$ (cf.\ Appendix~\ref{sect:elliptic_functions}), we have $\Phi'(\lambda)<0$ for $\lambda\in (0,\ell^{-2}\hat{\lambda})$. 
Moreover, $\Phi(\ell^{-2}\hat{\lambda})=0$ also follows since $\lambda=\ell^{-2}\hat{\lambda}$ implies that $q_1=q_2=\hat{q}$.
Thus we obtain \eqref{eq:comparison_arc_1stStep}, and equality holds if and only if $\Phi(\lambda)=0$, i.e., $\lambda=\ell^{-2}\hat{\lambda}$.

Next we consider the case $\lambda\ell^2 > \hat{\lambda}$.
Fix $n\geq n_{\lambda, \ell}$ arbitrarily. 
Setting $\lambda'=\frac{1}{n^2}\lambda$ and noting that then $q_1(\frac{\lambda \ell^2}{n^2}) = q_1(\lambda' \ell^2)$, we deduce from Lemma~\ref{lem:energy-formula} that $\mathcal{E}_\lambda[\gamma_{\rm sarc}^{\lambda,\ell,n}]=n^2 \mathcal{E}_{\lambda'}[\gamma_{\rm sarc}^{\lambda',\ell,1}]$ and
$\mathcal{E}_\lambda[\gamma_{\rm larc}^{\lambda,\ell,n}]=n^2 \mathcal{E}_{\lambda'}[\gamma_{\rm larc}^{\lambda',\ell,1}]$, respectively.
Moreover, noting that 
\begin{align}\label{eq:comparison_arc_lambda-hat}
\lambda'=\frac{1}{n^2}\lambda \leq \frac{1}{n_{\lambda,\ell}^2}\lambda \leq \ell^{-2}\hat{\lambda},
\end{align}
and applying \eqref{eq:comparison_arc_1stStep} to $\lambda=\lambda'$, we obtain 
\[\mathcal{E}_\lambda[\gamma_{\rm larc}^{\lambda,\ell,n}]=n^2 \mathcal{E}_{\lambda'}[\gamma_{\rm larc}^{\lambda',\ell,1}]\leq n^2 \mathcal{E}_{\lambda'}[\gamma_{\rm sarc}^{\lambda',\ell,1}]=\mathcal{E}_\lambda[\gamma_{\rm sarc}^{\lambda,\ell,n}].
\]
Equality in the above inequality holds if and only if $\lambda'=\ell^{-2}\hat{\lambda}$, i.e., equality holds for all the inequalities in \eqref{eq:comparison_arc_lambda-hat}, which is equivalent to $\sqrt{\frac{\lambda\ell^2}{\hat{\lambda}}}\in\mathbf{N}$ and $n=n_{\lambda,\ell}$.
\end{proof}

\begin{lemma}[Energy comparison of one-mode longer arc and one-mode loop]\label{lem:energy-comparison-arc_vs_loop}
If $\lambda>0$ and $\ell>0$ satisfy 
$\lambda\ell^2 \leq\hat{\lambda}$, then
\[\mathcal{E}_\lambda[\gamma_{\rm larc}^{\lambda,\ell,1}]<\mathcal{E}_\lambda[\gamma_{\rm loop}^{\lambda,\ell,1}],
\]
where the left-hand side is interpreted as $\mathcal{E}_\lambda[\gamma_{\rm larc}^{\lambda,\ell,1}] = \mathcal{E}_\lambda[\gamma_{\rm sarc}^{\lambda,\ell,1}]$ if $\lambda\ell^2 =\hat{\lambda}$.
\end{lemma}
\begin{proof}
We deduce from \eqref{eq:energy-larc-part2} that 
$\mathcal{E}_\lambda[\gamma_{\rm larc}^{\lambda,\ell,1}] = 2\sqrt{2} \sqrt{\lambda}h(q_{2,1})$ and $\mathcal{E}_\lambda[\gamma_{\rm loop}^{\lambda,\ell,1}] = 2\sqrt{2} \sqrt{\lambda}h(q_{3,1})$. 
Lemma~\ref{lem:property-h} and the fact that $\hat{q} \leq q_{2,1}<q_*<q_{3,1}$ imply that $h(q_{2,1})<h(q_{3,1})$, so that $\mathcal{E}_\lambda[\gamma_{\rm larc}^{\lambda,\ell,1}]<\mathcal{E}_\lambda[\gamma_{\rm loop}^{\lambda,\ell,1}]$. 
\end{proof}


On the other hand, the energy of $\gamma_{\rm larc}^{\lambda,\ell,n}$ with $n\geq2$ is higher than that of $\gamma_{\rm loop}^{\lambda,\ell,1}$:

\begin{lemma}[Energy comparison of higher-mode longer arc and one-mode loop]\label{lem:energy-comparison-arc_vs_loop_2}
Let $\lambda>0$, $\ell>0$, and $n\geq2$. 
Then, 
\[\mathcal{E}_\lambda[\gamma_{\rm larc}^{\lambda,\ell,n}]>\mathcal{E}_\lambda[\gamma_{\rm loop}^{\lambda,\ell,1}].
\]
\end{lemma}

We split the proof of Lemma~\ref{lem:energy-comparison-arc_vs_loop_2} into two lemmas. 
First we investigate how the energy $\mathcal{E}_\lambda[\gamma_{\rm larc}^{\lambda,\ell,n}]$ and $\mathcal{E}_\lambda[\gamma_{\rm loop}^{\lambda,\ell,1}]$ depends on $\lambda$. 
Here recall from Theorem~\ref{thm:classification-PPE} that if $\gamma_{\rm larc}^{\lambda,\ell,n}$ exists then we have necessarily $n\geq n_{\lambda,\ell}$, i.e., $\lambda \leq n^2\ell^{-2}\hat{\lambda}$.
Also recall from \eqref{eq:energy-larc-part2} and Definition \ref{def:sarc_larc_loop} that $\mathcal{E}_\lambda[\gamma_{\rm larc}^{\lambda,\ell,n}]=2\sqrt{2}n\sqrt{\lambda} h(q_2(\frac{\lambda\ell^2}{n^2}))$ and $\mathcal{E}_\lambda[\gamma_{\rm loop}^{\lambda,\ell,1}]=2\sqrt{2}\sqrt{\lambda} h(q_3(\lambda\ell^2))$.


\begin{lemma}\label{lem:Psi-decreasing}
Let $\ell>0$, $n\geq2$. 
Then, the map
\begin{align}\label{eq:def-Psi}
\Psi_{\ell,n}:\big(0, n^2\ell^{-2}\hat{\lambda}\big] \ni \lambda \mapsto n h\big(q_2(\tfrac{\lambda\ell^2}{n^2})\big) - h\big(q_3(\lambda\ell^2)\big)
\end{align}
is strictly decreasing on $(0,n^2\ell^{-2}\hat{\lambda}]$.
\end{lemma}
\begin{proof}
Noting that $q_2(\tfrac{\lambda\ell^2}{n^2})$ and $q_3(\lambda\ell^2)$ are given by $\lambda=\ell^{-2}n^2g(q_2(\tfrac{\lambda\ell^2}{n^2}))$ and $\lambda=\ell^{-2}g(q_3({\lambda\ell^2}))$, respectively, we compute 
\begin{align*}
    \Psi_{\ell,n}'(\lambda) &= nh'\big(q_2(\tfrac{\lambda\ell^2}{n^2})\big)\frac{d q_2(\tfrac{\lambda\ell^2}{n^2})}{d\lambda} - h'\big(q_3(\lambda\ell^2) \big)\frac{d q_3({\lambda\ell^2})}{d\lambda} \\
    &
    = nh'\big(q_2(\tfrac{\lambda\ell^2}{n^2})\big)\frac{\ell^2}{n^2g'(q_2(\tfrac{\lambda\ell^2}{n^2}))} - h'\big(q_3({\lambda\ell^2})\big)\frac{\ell^2}{g'(q_3(\lambda\ell^2))} \\
    &= -\frac{1}{n}\cdot\frac{\ell^2}{16(2q_2(\tfrac{\lambda\ell^2}{n^2})^2-1)^{\frac{3}{2}} (2\mathrm{E}(q_2(\tfrac{\lambda\ell^2}{n^2})) -\mathrm{K}(q_2(\tfrac{\lambda\ell^2}{n^2})))} \\
    &\quad +\frac{\ell^2}{16(2q_3(\lambda\ell^2)^2-1)^{\frac{3}{2}} (\mathrm{K}(q_3(\lambda\ell^2)) - 2\mathrm{E}(q_3(\lambda\ell^2)))}, 
\end{align*}
where in the last equality we used the derivative formulae \eqref{eq:diff-g} and \eqref{eq:diff-h}. 
Since $\lambda=\ell^{-2}n^2g(q_2(\tfrac{\lambda\ell^2}{n^2}))=\ell^{-2}g(q_3({\lambda\ell^2}))$ implies that $n(2q_2(\tfrac{\lambda\ell^2}{n^2})^2-1)^{\frac{1}{2}}(2\mathrm{E}(q_2(\tfrac{\lambda\ell^2}{n^2})) -\mathrm{K}(q_2(\tfrac{\lambda\ell^2}{n^2}))) = (2q_3(\lambda\ell^2)^2-1)^{\frac{1}{2}}(\mathrm{K}(q_3(\lambda\ell^2)) - 2\mathrm{E}(q_3(\lambda\ell^2)))$   
 we compute 
\begin{align*}
    \Psi_{\ell,n}'(\lambda)&=-\frac{1}{n}\cdot\frac{\ell^2}{16(2q_2(\tfrac{\lambda\ell^2}{n^2})^2-1)^{\frac{3}{2}} (2\mathrm{E}(q_2(\tfrac{\lambda\ell^2}{n^2})) -\mathrm{K}(q_2(\tfrac{\lambda\ell^2}{n^2})))} \\
    &\quad +\frac{\ell^2}{16n(2q_3(\lambda\ell^2)^2-1) (2q_2(\tfrac{\lambda\ell^2}{n^2})^2-1)^{\frac{1}{2}} (2\mathrm{E}(q_2(\tfrac{\lambda\ell^2}{n^2})) -\mathrm{K}(q_2(\tfrac{\lambda\ell^2}{n^2})))} \\
    &=\frac{\ell^2}{16n(2q_2(\tfrac{\lambda\ell^2}{n^2})^2-1)^{\frac{3}{2}}(2\mathrm{E}(q_2(\tfrac{\lambda\ell^2}{n^2})) -\mathrm{K}(q_2(\tfrac{\lambda\ell^2}{n^2})))}\Big(-1+\frac{2q_2(\tfrac{\lambda\ell^2}{n^2})^2-1}{2q_3(\lambda\ell^2)^2-1}\Big).
\end{align*}
Since $q_2(\tfrac{\lambda\ell^2}{n^2})<q_*<q_3(\lambda\ell^2)$, $\Psi_{\ell,n}'(\lambda)$ takes a negative value for each $\lambda \in (0,n^2 \ell^{-2} \hat{\lambda})$. 
The claim follows. 
\end{proof}

\begin{lemma}\label{lem:part-comparison}
For any $\ell>0$ and $n\geq2$, the map $\Psi_{\ell,n}$ defined by \eqref{eq:def-Psi} satisfies 
\[
\Psi_{\ell,n}(\lambda) >0 \quad \text{for all} \ \ \lambda \in \big(0, n^2\ell^{-2}\hat{\lambda} \big]. 
\] 
\end{lemma}
\begin{proof}
In view of Lemma~\ref{lem:Psi-decreasing}, it suffices to show that $\Psi_{\ell,n}(n^2\ell^{-2}\hat{\lambda})>0$. 
By definition $q_2(\frac{\lambda\ell^2}{n^2})=\hat{q}$ holds when $\lambda= n^2\ell^{-2}\hat{\lambda}$.
Throughout this proof write $\rho_n:=q_3(n^2\hat{\lambda})$ for short. 
Then, the problem is reduced to the positivity of $\Psi_{\ell,n}(n^2\ell^{-2}\hat{\lambda})=nh(\hat{q})-h(\rho_n)$.
Since $\hat{\lambda}=g(\hat{q})=n^{-2}g(\rho_n)$ yields (after taking square roots) that 
\begin{align}\label{eq:hat_q-rho_n-relation}
    \sqrt{2\rho_n^2-1}\mathrm{K}(\rho_n)=2\sqrt{2\rho_n^2-1}\mathrm{E}(\rho_n)+n\sqrt{2\hat{q}^2-1}(2\mathrm{E}(\hat{q})-\mathrm{K}(\hat{q})), 
\end{align} 
it follows that 
\begin{align*}
    h(\rho_n)&=\frac{(4\rho_n^2-3)\mathrm{K}(\rho_n)+2\mathrm{E}(\rho_n)}{\sqrt{2\rho_n^2-1}}
    =2\sqrt{2\rho_n^2-1}\mathrm{K}(\rho_n)+\frac{2\mathrm{E}(\rho_n)-\mathrm{K}(\rho_n)}{\sqrt{2\rho_n^2-1}} \\
    &=2\Big(2\sqrt{2\rho_n^2-1}\mathrm{E}(\rho_n)+n\sqrt{2\hat{q}^2-1}(2\mathrm{E}(\hat{q})-\mathrm{K}(\hat{q}))\Big) + \frac{2\mathrm{E}(\rho_n)-\mathrm{K}(\rho_n)}{\sqrt{2\rho_n^2-1}}.
\end{align*}
Moreover, using \eqref{eq:hat_q-rho_n-relation} again, and noting that $\frac{1}{\sqrt{2}} < \rho_n < 1$, we have
\[
\frac{2\mathrm{E}(\rho_n)-\mathrm{K}(\rho_n)}{\sqrt{2\rho_n^2-1}}
=-\frac{n\sqrt{2\hat{q}^2-1}}{2\rho_n^2-1} \big(2\mathrm{E}(\hat{q})-\mathrm{K}(\hat{q})\big)<-n\sqrt{2\hat{q}^2-1}\big(2\mathrm{E}(\hat{q})-\mathrm{K}(\hat{q})\big), 
\]
which leads to 
\begin{align*}
    h(\rho_n)&<4\mathrm{E}(\hat{q})+2n\sqrt{2\hat{q}^2-1}\mathrm{E}(\hat{q}) -n\sqrt{2\hat{q}^2-1}\mathrm{K}(\hat{q}), 
\end{align*}
where for the first term in the right-hand side we also used $\hat{q}<q_*<\rho_n<1$ and monotonicity of $\mathrm{E}(\cdot)$.
Thus we obtain 
\begin{align*}
    \Psi_{\ell,n}(n^2\ell^{-2}\hat{\lambda})&=n \Big(2\sqrt{2\hat{q}^2-1}\mathrm{K}(\hat{q})+\frac{2\mathrm{E}(\hat{q})-\mathrm{K}(\hat{q})}{\sqrt{2\hat{q}^2-1}}\Big) - h(\rho_n) \\
    &>3n\sqrt{2\hat{q}^2-1}\mathrm{K}(\hat{q})+n\frac{2\mathrm{E}(\hat{q})-\mathrm{K}(\hat{q})}{\sqrt{2\hat{q}^2-1}}-4\mathrm{E}(\hat{q})-2n\sqrt{2\hat{q}^2-1}\mathrm{E}(\hat{q})\\
    &=\frac{1}{\sqrt{2\hat{q}^2-1}}\Big(2n(3\hat{q}^2-2)\mathrm{K}(\hat{q})+4\big(n(1-\hat{q}^2)-\sqrt{2\hat{q}^2-1}\big)\mathrm{E}(\hat{q})\Big).
\end{align*}
Set $d_n:=2n(3\hat{q}^2-2)\mathrm{K}(\hat{q})+4(n(1-\hat{q}^2)-\sqrt{2\hat{q}^2-1})\mathrm{E}(\hat{q})$ and hereafter we show that $d_n>0$. 
Since $\hat{q}$ is a solution of $f(q)=0$, $\hat{q}$ satisfies $(-4\hat{q}^4+5\hat{q}^2-1)\mathrm{K}(\hat{q})=(-8\hat{q}^4+8\hat{q}^2-1)\mathrm{E}(\hat{q})$.
This yields 
\begin{align*}
    d_n&=2n(3\hat{q}^2-2)\frac{-8\hat{q}^4+8\hat{q}^2-1}{-4\hat{q}^4+5\hat{q}^2-1}\mathrm{E}(\hat{q})+4\big(n(1-\hat{q}^2)-\sqrt{2\hat{q}^2-1}\big)\mathrm{E}(\hat{q}) \\
    &=2n \frac{-16\hat{q}^6+22\hat{q}^4-7\hat{q}^2}{-4\hat{q}^4+5\hat{q}^2-1}\mathrm{E}(\hat{q}) -4\sqrt{2\hat{q}^2-1}\mathrm{E}(\hat{q}).
\end{align*}
Note that $-4\hat{q}^4+5\hat{q}^2-1=(4\hat{q}^2-1)(1-\hat{q}^2)>0$ since $\hat{q}>{1}/{\sqrt{2}}$.
By Lemma~\ref{lem:hat-q-character} and the fact that $[\frac{3}{5},\frac{2}{3}] \ni m \mapsto -16m^3+22m^2-7m$ is a positive function, we also find that $-16\hat{q}^6+22\hat{q}^4-7\hat{q}^2>0$ and hence 
\[
d_n\geq d_2 = \frac{1}{-4\hat{q}^4+5\hat{q}^2-1}\Big(4(-16\hat{q}^6+22\hat{q}^4-7\hat{q}^2) -4\sqrt{2\hat{q}^2-1}(-4\hat{q}^4+5\hat{q}^2-1)\Big)\mathrm{E}(\hat{q}).
\]
In order to deduce that $d_2>0$, it suffices to show the positivity of
\[
\psi(m):=4(-16m^3+22m^2-7m)-4\sqrt{2m-1}(-4m^2+5m-1), \quad m\in[\tfrac{3}{5}, \tfrac{2}{3}],
\]
since $\hat{q}^2\in(\tfrac{3}{5}, \tfrac{2}{3})$.
The fact that $\sqrt{2m-1} \leq 2m-\frac{3}{4}$ yields 
\begin{align*}
\psi(m)&\geq4(-16m^3+22m^2-7m)-4(2m-\tfrac{3}{4})(-4m^2+5m-1) \\
&= -32m^3+36m^2-5m-3.
\end{align*}
Notice that $m\mapsto -32m^3+36m^2-5m-3$ is strictly increasing in $[\tfrac{3}{5}, \tfrac{2}{3}]$. 
Indeed, the derivative $m \mapsto -96m^2 + 72 m -5$ has roots $m_{1,2} = \frac{3 \pm \sqrt{\frac{17}{3}}}{8}$, implying that $m_1 < 0$ and $m_2 > \frac{2}{3}$.
Due to this monotonicity we infer $\psi(m)\geq -32(\tfrac{3}{5})^3+36(\tfrac{3}{5})^2-5\cdot\tfrac{3}{5}-3=\frac{6}{125}>0$.
The positivity of $\Psi_{\ell,n}(n^2\ell^{-2}\hat{\lambda})$ follows and the proof is complete.
\end{proof}

Now we turn to the proof of Lemma~\ref{lem:energy-comparison-arc_vs_loop_2}. 
\begin{proof}[Proof of Lemma~\ref{lem:energy-comparison-arc_vs_loop_2}]
The combination of the energy representation formula as in \eqref{eq:energy-larc-part2} and Lemma~\ref{lem:part-comparison} yields 
\begin{align*}
\mathcal{E}_\lambda[\gamma_{\rm larc}^{\lambda,\ell,n}] - \mathcal{E}_\lambda[\gamma_{\rm loop}^{\lambda,\ell,1}]
&=2\sqrt{2}\sqrt{\lambda} \Psi_{\ell,n}(\lambda)>0, 
\end{align*}
which completes the proof.
\end{proof}

In addition, as follows one can observe the reversal of the order of the energy between $\gamma_{\rm sarc}^{\lambda,\ell,1}$ and  $\gamma_{\rm loop}^{\lambda,\ell,1}$, depending on $\lambda$. 
\begin{lemma}[Energy comparison of one-mode shorter arc and one-mode loop]\label{lem:energy-comparison-sarc_vs_loop}
Let $\lambda>0$ and $\ell>0$ satisfy $\lambda\ell^2 <\hat{\lambda}$.
Then, there exists a unique $\lambda_\dagger\in (0,\hat{\lambda})$ (independent of $\lambda$ and $\ell$) such that
\[
\mathcal{E}_\lambda[\gamma_{\rm sarc}^{\lambda,\ell,1}]>\mathcal{E}_\lambda[\gamma_{\rm loop}^{\lambda,\ell,1}] \ \  \text{if} \ \ \lambda \ell^2 <\lambda_\dagger \quad \text{and} \quad 
\mathcal{E}_\lambda[\gamma_{\rm sarc}^{\lambda,\ell,1}]<\mathcal{E}_\lambda[\gamma_{\rm loop}^{\lambda,\ell,1}] \ \  \text{if} \ \ \lambda \ell^2 >\lambda_\dagger.
\]
In addition, $\mathcal{E}_\lambda[\gamma_{\rm sarc}^{\lambda,\ell,1}]=\mathcal{E}_\lambda[\gamma_{\rm loop}^{\lambda,\ell,1}]$ if $\lambda \ell^2=\lambda_\dagger$.
\end{lemma}
\begin{proof} 
For $\ell>0$ we define $\phi_\ell:(0,\hat{\lambda}]\to\mathbf{R}$ by
\[
\phi_\ell(\lambda):=\mathcal{E}_\lambda[\gamma_{\rm sarc}^{\lambda,\ell,1}]-\mathcal{E}_\lambda[\gamma_{\rm loop}^{\lambda,\ell,1}].
\]
From \eqref{eq:energy-PPE} (and the fact that $q_{i,1} = q_i(\lambda\ell^2)$ for $i = 1,
3$) one infers that
\begin{equation}\label{eq:phi_scaling_ell}
    \phi_\ell(\lambda) = \frac{1}{\ell} \phi_1(\lambda \ell^2).
\end{equation}
Since the energy order depends only on the sign of $\phi_\ell$ (and thus of $\phi_1(\cdot \ell^2)$) the above formula allows us to infer the claim from the study of the special case of $\ell= 1$. 
The fact that $\gamma_{\rm sarc}^{\hat{\lambda},1,1}  = \gamma_{\rm larc}^{\hat{\lambda},1,1}$ and Lemma~\ref{lem:energy-comparison-arc_vs_loop} yield that $\phi_1(\hat{\lambda}) =\mathcal{E}_{\hat{\lambda}}[\gamma_{\rm larc}^{\hat{\lambda},1,1}]-\mathcal{E}_{\hat{\lambda}}[\gamma_{\rm loop}^{\hat{\lambda},1,1}] < 0$. 
Moreover, noting that $q_1(\lambda)\to 1/\sqrt{2}$ and $q_3(\lambda)\to q_*$ as $\lambda\to0$, we deduce from the energy formulae \eqref{eq:energy-PPE} that
\begin{align*}
    \lim_{\lambda\to0}\mathcal{E}_\lambda[\gamma_{\rm sarc}^{\lambda,1,1}] &= 8 \big(2\mathrm{E}(\tfrac{1}{\sqrt{2}})-\mathrm{K}(\tfrac{1}{\sqrt{2}}) \big)^2>0, \\
    \lim_{\lambda\to0}\mathcal{E}_\lambda[\gamma_{\rm loop}^{\lambda,1,1}] &= 0,
\end{align*}
from which it follows that $\phi_1(0+)>0$. 
Next we show that $\phi_1$ is strictly decreasing on $(0,\hat{\lambda}]$.
As in \eqref{eq:diff-Phi_energy}, we compute
\begin{align*}
\phi_1'(\lambda)=&\,\frac{dq}{d\lambda}\Big|_{q=q_1}\frac{d}{dq}\Big(8|2\mathrm{E}(q)-\mathrm{K}(q)|\big((4q^2-3)\mathrm{K}(q)+2\mathrm{E}(q) \big) \Big)\Big|_{q=q_1} \\
    &- \frac{dq}{d\lambda}\Big|_{q=q_3}\frac{d}{dq}\Big(8|2\mathrm{E}(q)-\mathrm{K}(q)|\big((4q^2-3)\mathrm{K}(q)+2\mathrm{E}(q) \big) \Big)\Big|_{q=q_3} \\
    =&\,\frac{\ell^2}{g'(q_1)}\frac{16}{q_1(1-q_1^2)}f(q_1)\mathrm{K}(q_1)
    +\frac{\ell^2}{g'(q_3)}\frac{16}{q_3(1-q_3^2)}f(q_3)\mathrm{K}(q_3), \\
    =&\,\ell^2\Big(\frac{\mathrm{K}(q_1)}{2\mathrm{E}(q_1)-\mathrm{K}(q_1)} - \frac{\mathrm{K}(q_3)}{\mathrm{K}(q_3)-2\mathrm{E}(q_3)} \Big). 
\end{align*}
Since $\lambda=g(q_1(\lambda))=g(q_3(\lambda))$  yield 
$2\mathrm{E}(q_1)-\mathrm{K}(q_1)= \frac{1}{2\sqrt{2}}\sqrt{\lambda}(2q_1^2-1)^{-\frac{1}{2}}$ and 
$ \mathrm{K}(q_3) -2\mathrm{E}(q_3)= \frac{1}{2\sqrt{2}}\sqrt{\lambda}(2q_3^2-1)^{-\frac{1}{2}}$, 
we obtain 
\[
    \phi_1'(\lambda)=\frac{ 2\sqrt{2}}{ \sqrt{\lambda}}\bigg(\mathrm{K}(q_1)\sqrt{2q_1^2-1} - \mathrm{K}(q_3)\sqrt{2q_3^2-1} \bigg).
\]
Then, in view of the fact that the map $q\mapsto \mathrm{K}(q)\sqrt{2q^2-1}$ is strictly increasing, we find that the right-hand side in the above equation takes a negative value for all $\lambda\in(0,\hat{\lambda})$.
This together with $\phi_1(0+)>0$ and $\phi_1(\hat{\lambda})<0$ implies that there exists a unique $\lambda_\dagger\in(0,\hat{\lambda})$ such that $\phi_1(\lambda_\dagger)=0$. 
By monotonicity of $\phi_1$ and \eqref{eq:phi_scaling_ell} it follows that if $\lambda\ell^2 < \lambda_\dagger$, then
\[
\mathcal{E}_\lambda[\gamma_{\rm sarc}^{\lambda,\ell,1}]-\mathcal{E}_\lambda[\gamma_{\rm loop}^{\lambda,\ell,1}]=\phi_\ell(\lambda)=\frac{1}{\ell}\phi_1(\lambda\ell^2)>\frac{1}{\ell}\phi_1(\lambda_\dagger) =0.
\]
Similarly the remaining assertions follow.
The proof is complete.
\end{proof}

\begin{remark}
This remark summarizes the insights gained in the previous lemmas in the important special case of small $\lambda$, i.e. $\lambda \ell^2 < \hat{\lambda}$. Notice that in this case one has $n_{\lambda, \ell} = 1$. By Lemma \ref{lem:energy-comparison-sarc_vs_larc} we have $\mathcal{E}_\lambda[\gamma_{\rm larc}^{\lambda,\ell,1}] \leq \mathcal{E}_\lambda[\gamma_{\rm sarc}^{\lambda,\ell,1}]$ and by Lemma \ref{lem:energy-comparison-arc_vs_loop} we have $\mathcal{E}_\lambda[\gamma_{\rm larc}^{\lambda,\ell,1}] \leq \mathcal{E}_\lambda[\gamma_{\rm loop}^{\lambda,\ell,1}]$. We infer
    \begin{itemize}
        \item $\gamma_{\rm larc}^{\lambda, \ell, 1}$ has minimal energy of all penalized pinned elasticae in $A_\ell$ (except for the line).
        \item $\gamma_{\rm sarc}^{\lambda, \ell, 1}, \gamma_{\rm loop}^{\lambda, \ell, 1} $ have a larger energy than $\gamma_{\rm larc}^{\lambda, \ell, 1}$, but their order depends on $\lambda$, cf.\ Lemma~\ref{lem:energy-comparison-sarc_vs_loop}.  More precisely we have shown that the order changes once at $\ell^{-2}\lambda_\dagger$. From Figure  \ref{fig:EnergyComp}, which shows the special case of $\ell = 1$, one can read off that $ \lambda_\dagger \simeq  0.32241$.
        \item Since increasing the mode $n$ makes the energy larger (cf.\ Lemma \ref{lem:energy-comparison-mode}) we infer that all the elasticae with higher modes are not minimal.  
    \end{itemize}
\end{remark}

\begin{figure}[h!]
    \centering
    \includegraphics[width=0.5\linewidth]{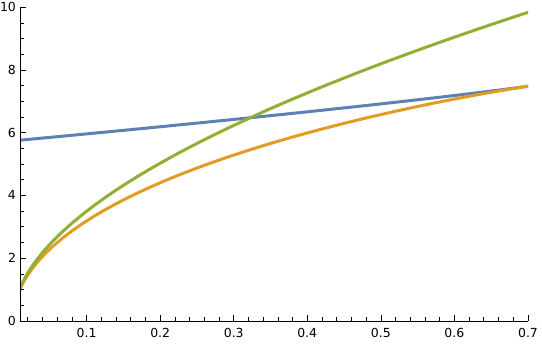}
    \caption{Plots of $\lambda \mapsto \mathcal{E}_\lambda[\gamma_{\rm sarc}^{\lambda,1,1}]$ (blue), $\lambda \mapsto \mathcal{E}_\lambda[\gamma_{\rm larc}^{\lambda,1,1}]$ (yellow) and $\lambda \mapsto \mathcal{E}_\lambda[\gamma_{\rm loop}^{\lambda,1,1}]$ (green) for $\lambda \in (0.02,0.7)$ }
    \label{fig:EnergyComp}
\end{figure}

We close this section by the proof of Theorem~\ref{thm:unique_nontrivial_PPE}. 

\begin{proof}[Proof of Theorem~\ref{thm:unique_nontrivial_PPE}]
This is a direct consequence of Theorem~\ref{thm:classification-PPE} and Lemmas~\ref{lem:energy-comparison-mode}, \ref{lem:energy-comparison-sarc_vs_larc}, \ref{lem:energy-comparison-arc_vs_loop}, and \ref{lem:energy-comparison-arc_vs_loop_2}. 
\end{proof}


\section{Elastic flow}\label{sect:elastic-flow}

In this section we apply our classification results in Theorem~\ref{thm:classification-PPE} and energy-comparison results in Lemmas~\ref{lem:energy-comparison-sarc_vs_larc} and \ref{lem:energy-comparison-arc_vs_loop_2} to the asymptotic behavior of the $\lambda$-elastic flow \eqref{eq:elastic_flow} under the Navier boundary conditions, cf.\ \eqref{eq:Navier-BC_flow}. 
It is already known that the solution to the flow subconverges to a stationary solution, which satisfies 
\begin{align}\label{eq:stationary}
\begin{cases}
  2\partial_s^2k+k^3-\lambda k=0 \quad \text{in} \ \ (0,1), \\
    \gamma(0)=(0,0), \ \gamma(1) =(\ell,0), \\ 
    \kappa(0)=\kappa(1)=0.
\end{cases}
\end{align}
We first recall the following statement (see e.g.\ \cite[Section 2]{NO_2014} or \cite[Proposition 5.2]{MPP21}): 
\begin{proposition}[Long-time existence and subconvergence]\label{prop:MPP21_Prop5.2}
Let $\gamma_0 \in A_\ell$ be a smoothly immersed curve such that $\kappa(0)=\kappa(1)=0$. 
Then, there is a unique global-in-time smooth solution $\gamma:[0,1]\times [0,\infty)\to\mathbf{R}^2$ to the initial value problem, 
\begin{align*}
\begin{cases}
    \partial_t \gamma = -2\nabla^2_s \kappa - |\kappa|^2\kappa + \lambda \kappa \quad &\text{in} \ \ [0,1]\times(0,\infty),  \\ 
    \gamma(0,t)=(0,0), \ \ \gamma(1,t)=(\ell,0) & \text{on} \ \ (0,\infty), \\
    \kappa(0,t)=\kappa(1,t)=0, & \text{on} \ \ (0,\infty), \\
    \gamma(x,0)=\gamma_0(x) & \text{on} \ \ [0,1].
\end{cases}
\end{align*}
Moreover, for any sequence $t_j\to\infty$ there exist a subsequence $t_{j'} \to \infty$ and a smoothly immersed curve $\gamma_\infty$ satisfying \eqref{eq:stationary} such that $\gamma(\cdot,t_j)$ converges smoothly to $\gamma_\infty$ up to reparametrization.
\end{proposition}
The last subconvergence statement in Proposition~\ref{prop:MPP21_Prop5.2} is slightly different from the original statement, but an inspection of the proofs in the above references immediately implies the above formulation. 
In view of Lemma~\ref{lem:EL-PPE}, critical points of $\mathcal{E}_\lambda$ in $A_\ell$ in the sense of Definition~\ref{def:PPE} can be characterized by \eqref{eq:stationary}.
Thus, stationary solutions of \eqref{eq:stationary} correspond to critical points of $\mathcal{E}_\lambda$ in $A_\ell$. 
Therefore, Theorem~\ref{thm:classification-PPE} can be also regarded as the complete classification of solutions to \eqref{eq:stationary}. 

The classification and assumption \eqref{eq:embedded-threshold} significantly reduce the candidates of trajectories of the elastic flow. 
In view of Lemmas~\ref{lem:energy-comparison-sarc_vs_larc} and \ref{lem:energy-comparison-arc_vs_loop_2}, a non-trivial critical point of $\mathcal{E}_\lambda$ in $A_\ell$ whose energy is less than that of $\gamma_{\rm loop}^{\lambda,\ell,1}$ is either $\bar{\gamma}_{\rm sarc}^{\lambda,\ell,1}$ or $\bar{\gamma}_{\rm larc}^{\lambda,\ell,1}$, where $\bar{\gamma}_{\rm sarc}^{\lambda,\ell,1}:[0,1]\to\mathbf{R}^2$ and $\bar{\gamma}_{\rm larc}^{\lambda,\ell,1}:[0,1]\to\mathbf{R}^2$ denotes the constant-speed reparametrization of $\gamma_{\rm sarc}^{\lambda,\ell,1}$ and $\gamma_{\rm larc}^{\lambda,\ell,1}$, respectively. 
Thus, recalling the fundamental property $\frac{d}{dt}\mathcal{E}_\lambda[\gamma(\cdot,t)]\leq0$, we see that any limit curve $\gamma_\infty$ must (after reparametrization) belong to
\begin{align}\label{eq:omega_limit_set}
    \omega&:=\Set{ \gamma \in A_\ell | \gamma \text{ is either $\gamma_{\rm seg}$, $\bar{\gamma}_{\rm larc}^{\lambda,\ell,1}$, or $\bar{\gamma}_{\rm sarc}^{\lambda,\ell,1}$}}, 
\end{align}
where $\gamma_{\rm seg}:[0,1]\to\mathbf{R}^2$ denotes the unique line segment in $A_\ell$.  We interpret $\omega=\{ \gamma_{\rm seg}\}$ if $\lambda\ell^2 > \hat{\lambda}$ since in this case $\bar{\gamma}_{\rm larc}^{\lambda,\ell,1}$ and $\bar{\gamma}_{\rm sarc}^{\lambda,\ell,1}$ are absent.

To study embeddedness of the elastic flow it will be important to investigate whether a curve in $\omega$ is embedded. 

\begin{lemma}\label{lem:embeddedness_omega} 
Every curve $\gamma\in\omega$ is embedded. 
\end{lemma}
\begin{proof}
The case $\lambda\ell^2>\hat{\lambda}$ is trivial, so hereafter we consider $0<\lambda\ell^2 \leq \hat{\lambda}$. 
Then, it suffices to check embeddedness of $\gamma_{\rm sarc}^{\lambda,\ell,1}$ and $\gamma_{\rm larc}^{\lambda,\ell,1}$.
Since the proof is completely parallel, we only consider $\gamma_{\rm sarc}^{\lambda,\ell,1}$. 
Set $(X(s), Y(s)):=\gamma_{\rm sarc}^{\lambda,\ell,1}(s)$. 
We deduce from \eqref{def:gamma_sarc-PPE} that 
$
X(s)=\frac{1}{\alpha} ( 2\mathrm{E}(\am(\alpha s-\mathrm{K}(q),q),q) + 2\mathrm{E}(q) - \alpha s )
$
with $q:=q_1({\lambda\ell^2})$ and $\alpha:=\frac{2}{\ell}(2\mathrm{E}(q)-\mathrm{K}(q))$. 
As in \eqref{eq:X_+_compute} we obtain
\[
X'(s)=1-2q^2\sn(\alpha s-\mathrm{K}(q),q)^2.
\]
Since $L[\gamma_{\rm sarc}^{\lambda,\ell,1}]=2\mathrm{K}(q)/\alpha$, we see that $X'(s)$ is increasing for $s\in (0,{L}/{2})$.
Thus $X$ is convex in $[0,{L}/{2}]$. 
This together with the fact that $X(0)=0$ and $X(\frac{L}{2})=\ell/2$ implies that $X(s)<\ell/2$ holds for all $s\in(0,L/2)$. 
By reflection symmetry (cf.\ \eqref{eq:symmetry-PPE}), it suffices to check embeddedness of $\gamma_{\rm sarc}^{\lambda,\ell,1}(s)$ for $s\in(0,L/2)$, which immediately follows from monotonicity of $Y(s)=\frac{2q}{\alpha}\cn(\alpha s-\mathrm{K}(q))$ on $(0,L/2)$.
\end{proof}

We are now ready to prove Theorem~\ref{thm:embeddedness}.

\begin{proof}[Proof of Theorem~\ref{thm:embeddedness}]
    Let $\gamma(\cdot, t)$ be the smooth evolution of the $\lambda$-elastic flow in $A_\ell$ with initial datum $\gamma_0$. We prove the theorem by contradiction. If the asserted time $t_0$ does not exist, then one can find a sequence $t_j \rightarrow \infty$ such that $\gamma(\cdot, t_j)$ is not embedded.
    By Proposition \ref{prop:MPP21_Prop5.2} we can find a subsequence $t_{j'} \rightarrow \infty$ and reparametrizations $\phi_{j'} : [0,1] \rightarrow [0,1]$ such that $\gamma(\phi_{j'}(\cdot),t_{j'})$ converges smoothly to some curve $\gamma_\infty$ satisfying \eqref{eq:stationary}. By \eqref{eq:omega_limit_set}, we can choose the reparametrizations $\phi_{j'}$ in such a way that $\gamma_\infty$ must belong to $\omega$. In particular, $\gamma_\infty$ is embedded by Lemma \ref{lem:embeddedness_omega}. Since the set of embedded curves is open in the $C^1$-topology, (which can e.g. be seen like in \cite[Lemma 4.3]{MR23}) and $\gamma(\phi_{j'}(\cdot),t_{j'}) \rightarrow \gamma_\infty$ in the $C^1$-topology, we obtain that $\gamma(\phi_{j'}(\cdot),t_{j'})$ must be embedded for some large $j'$. Since reparametrizations do not affect the embeddedness of a curve, we also have that $\gamma(\cdot, t_{j'})$ must be embedded.  A contradiction. 
\end{proof}

\begin{remark}[Extinction of a self-intersection]\label{rem:extinction_loop}
Finally we remark that the energy threshold \eqref{eq:embedded-threshold} is indeed undercut by curves that have a self-intersection. Theorem~\ref{thm:embeddedness} yields then that for these curves, all the self-intersections become extinct in finite time. Here we give an explicit construction of nonembedded curves satisfying \eqref{eq:embedded-threshold}. 
By the instability of $\gamma_{\rm loop}^{\lambda,\ell,1}$ (cf.\ Lemma~\ref{lem:2nd-derivative}), we can find some $q\in(q_*,1)$ such that $\mathcal{E}_\lambda[\gamma_w(\cdot,q)]<C_{\lambda,\ell}$.
In particular, since for $q\in(q_*,1)$ the curve $\gamma_w(\cdot,q)$ has a self-intersection, the $\lambda$-elastic flow with initial datum $\gamma_0=\gamma_w(\cdot,q)$ has a self-intersection at $t=0$ (see also Remark~\ref{rem:smooth_and_loop}).
However, by Theorem~\ref{thm:embeddedness} there is a time $t_0>0$ such that such the $\lambda$-elastic flow possesses no self-intersection for all $t\geq t_0$.
\end{remark}

\appendix

\section{Elliptic integrals and functions}\label{sect:elliptic_functions}
In this article we have used the \emph{elliptic integrals}
 \begin{equation*}
     \mathrm{F}(x,q) := \int_0^x \frac{1}{\sqrt{1-q^2\sin^2(\theta) }} \; \mathrm{d}\theta \qquad \textrm{and} \qquad \mathrm{E}(x,q) := \int_0^x \sqrt{1-q^2\sin^2(\theta)} \; \mathrm{d}\theta 
 \end{equation*}
 for $x \in \mathbf{R}$ and $q \in (0,1)$. 
 Further we define $\mathrm{K}(q) := \mathrm{F}(\frac{\pi}{2},q)$ and $\mathrm{E}(q) := \mathrm{E}(\frac{\pi}{2},q)$.
It is known that $[0,1) \ni q\mapsto \mathrm{K}(q)$ and $[0,1]\ni q \mapsto \mathrm{E}(q)$ are strictly increasing and strictly decreasing, respectively.
More precisely, one has
\begin{align} \label{eq:diff-elliptic-int}
\mathrm{K}'(q)=\frac{\mathrm{E}(q)}{q(1-q^2)} - \frac{\mathrm{K}(q)}{q}>0, \quad \mathrm{E}'(q)= \frac{\mathrm{E}(q)}{q} - \frac{\mathrm{K}(q)}{q}<0.
\end{align}
This monotonicity leads to the following

\begin{lemma}\label{lem:elliptic_2E-K}
The function $Q:[0,1)\ni q\mapsto 2\mathrm{E}(q)-\mathrm{K}(q)$ is strictly decreasing, and satisfies $Q(0)=1$, $\lim_{q\to 1}Q(q)=-\infty$.
In particular, there exists a unique $q_*\in(0,1)$ such that $2\mathrm{E}(q_*)-\mathrm{K}(q_*)=0$.
\end{lemma}

The constant $q_*$ stands for the modulus of the so-called figure-eight elastica; more precisely, the signed curvature of the figure-eight elastica is given by $k(s)=2q_*\cn(s,q_*)$, up to similarity and reparametrization (cf.\ \cite[Definition 5.1]{MR23}). 

We also mention a useful formula to investigate the bending energy of the so-called wavelike elasticae: 
\begin{align} \label{eq:Ydcds_Lem2.6}
\frac{d}{dq}\Big(q^2\mathrm{K}(q)-\mathrm{K}(q)+\mathrm{E}(q) \Big) = q\mathrm{K}(q).
\end{align}
This follows from a straightforward calculation (cf.\ \cite[Lemma 2.6]{Ydcds}).

Next we recall the Jacobian elliptic functions $\cn,\sn$. 

\begin{definition}[Elliptic functions] \label{def:Jacobi-function}
We define the \emph{Jacobian amplitude function} $\am(x,q)$ by the inverse function of $F(x,q)$, so that
\[
x= \int_0^{\am(x,q)} \frac{1}{ \sqrt{1-q^2\sin^2\theta} }\; \mathrm{d}\theta \quad \text{for}\ x\in\mathbf{R}.
\]
For $q\in[0,1)$, the \emph{Jacobian elliptic functions} are given by
\begin{align*}
\cn(x, q):= \cos \am(x,q), \qquad \sn(x, q):= \sin \am(x,q), \qquad x\in \mathbf{R}.
\end{align*}
\end{definition}

The Jacobian elliptic functions have the following fundamental properties. 
\begin{proposition} \label{prop:cn}
Let $\cn(\cdot,q)$ and $\sn(\cdot,q)$ be the elliptic functions with modulus $q \in [0,1)$. 
Then, $\cn(\cdot,q)$ is an even $2\mathrm{K}(q)$-antiperiodic function on $\mathbf{R}$ and, in $[0,2\mathrm{K}(q)]$, strictly decreasing from $1$ to $-1$. 
Further, $\sn(\cdot,q)$ is an odd $2\mathrm{K}(q)$-antiperiodic function and in $[-\mathrm{K}(q),\mathrm{K}(q)]$ strictly increasing from $-1$ to $1$.
\end{proposition}

We also collect some integral formulae used in this paper.
For $q\in(0,1)$, 
\begin{align}
    \int \cn(\xi,q)\;\mathrm{d}\xi=\frac{1}{q}\arcsin{\big(q\sn(x,q) \big)} +C, \label{eq:int-cn} \\
    \int \big(1-2q^2\sn(\xi,q)^2\big)\;\mathrm{d}\xi = 2\mathrm{E}(\am(x,q),q)-x +C, \label{eq:int-1-sn^2} \\
    \int \sn(\xi,q)\sqrt{1-q^2\sn(\xi,q)^2}\;\mathrm{d}\xi = -\cn(x,q) +C. \label{eq:int-sn(1-q^2sn^2)}
\end{align}

\section{Technical proofs}\label{sect:proof-fgeh}

\begin{proof}[Proof of \eqref{eq:SingleTermComparison}] 
   In this proof we use the shorthand notation $q_1 := q_1(\lambda \ell^2)$ and $q_2 := q_2(\lambda \ell^2)$. Recall that ${1}/{\sqrt{2}}< q_1 < q_2< q_*$. Also define $\alpha_i := \frac{\sqrt{\lambda}}{\sqrt{2} \sqrt{2q_i^2-1}}$ for $i= 1,2$, cf.\ \eqref{eq:lambda-Linner}. We first show the inequality $L[\gamma_{\rm sarc}^{\lambda,\ell,1}] <  L[\gamma_{\rm larc}^{\lambda,\ell,1}]$. To this end we infer from Definition \ref{def:sarc_larc_loop} that 
   \begin{align*}
       L[\gamma_{\rm sarc}^{\lambda,\ell,1}]  = \ell \frac{ \mathrm{K}(q_1)}{2\mathrm{E}(q_1)- \mathrm{K}(q_1)}= \ell \frac{ 1}{2\frac{\mathrm{E}(q_1)}{\mathrm{K}(q_1)}- 1}, \\
       L[\gamma_{\rm larc}^{\lambda,\ell,1}]  = \ell \frac{ \mathrm{K}(q_2)}{2\mathrm{E}(q_2)- \mathrm{K}(q_2)}= \ell \frac{ 1}{2\frac{\mathrm{E}(q_2)}{\mathrm{K}(q_2)}- 1}.
   \end{align*}
       Notice that by \cite[Lemma B.4]{MR23} the function $q \mapsto 2 \frac{\mathrm{E}(q)}{\mathrm{K}(q)} - 1$ is decreasing and positive on $(0,q^*)$. The desired inequality follows then immediately from the previous two identities. Now we turn to the inequality  $B[\gamma_{\rm sarc}^{\lambda,\ell,1}] > B[\gamma_{\rm larc}^{\lambda,\ell,1}]$. Proceeding as in \eqref{eq:computationOFB} one can find 
    \begin{equation*}
        B[\gamma_{\rm sarc}^{\lambda,\ell,1}] = 8 \alpha_1 (q_1^2 \mathrm{K}(q_1) - \mathrm{K}(q_1) + \mathrm{E}(q_1) )   = 2\sqrt{2}\sqrt{\lambda} \frac{q_1^2\mathrm{K}(q_1)- \mathrm{K}(q_1) + \mathrm{E}(q_1)}{\sqrt{2q_1^2-1}}.
    \end{equation*}
    Similarly one finds
    $
        B[\gamma_{\rm larc}^{\lambda,\ell,1}] = 2\sqrt{2}\sqrt{\lambda} \frac{q_2^2\mathrm{K}(q_2)- \mathrm{K}(q_2) + \mathrm{E}(q_2)}{\sqrt{2q_2^2-1}}.
$
    Hence it suffices to show $\eta: (\frac{1}{\sqrt{2}},q_*) \rightarrow \mathbf{R}, \eta(q):= \frac{q^2\mathrm{K}(q)-\mathrm{K}(q)+\mathrm{E}(q)}{\sqrt{2q^2-1}}$ is decreasing on $(\frac{1}{\sqrt{2}},q_*)$.
    This is easily seen by computing with the aid of \eqref{eq:Ydcds_Lem2.6} that 
    \begin{align*}
        \eta'(q) & = \frac{q\mathrm{K}(q)}{\sqrt{2q^2-1}} - \frac{2q (q^2\mathrm{K}(q)-\mathrm{K}(q)+\mathrm{E}(q))}{(2q^2-1)^\frac{3}{2}}  \\ & = \frac{q\mathrm{K}(q)-2q\mathrm{E}(q)}{(2q^2-1)^\frac{3}{2}} = - \frac{q}{(2q^2-1)^\frac{3}{2}} (2\mathrm{E}(q)-\mathrm{K}(q)) < 0 .\qedhere
    \end{align*}

    \end{proof}

\begin{proof}[Proof of Lemma~\ref{lem:property-f}]
We first check that $f(\frac{1}{\sqrt{2}})>0$.
Indeed, by Lemma~\ref{lem:elliptic_2E-K} it follows that 
\begin{align}\label{eq:f-positive}
    f(\tfrac{1}{\sqrt{2}}) &= -\frac{1}{2} \mathrm{K}(\tfrac{1}{\sqrt{2}})+ \mathrm{E}(\tfrac{1}{\sqrt{2}}) = \frac{1}{2} \big(2\mathrm{E}(\tfrac{1}{\sqrt{2}}) - \mathrm{K}(\tfrac{1}{\sqrt{2}}) \big) >0.
\end{align}
Next we show that $f(q)<0$ if $q\in [q_*,1)$. 
Setting
\[
a(q):=4q^4-5q^2+1 = (4q^2-1)(q^2-1), \quad b(q):=-8q^4+8q^2-1, \quad q\in [\tfrac{1}{\sqrt{2}},1], 
\]
we can rewrite $f(q)=a(q)\mathrm{K}(q)+b(q)\mathrm{E}(q)$. 
By Lemma~\ref{lem:elliptic_2E-K}, $2\mathrm{E}(q)<\mathrm{K}(q)$ holds for $q\in [q_*,1)$. 
This together with the fact that $a$ is negative on $[\tfrac{1}{\sqrt{2}},1]$ implies that, for $q\in[q_*,1)$, 
\begin{align*}
f(q) = a(q)\mathrm{K}(q)+b(q)\mathrm{E}(q) < 2a(q)\mathrm{E}(q)+b(q)\mathrm{E}(q) = -(2q^2-1)\mathrm{E}(q). 
\end{align*}
Therefore we obtain 
\begin{align} \label{eq:f-negative}
    f(q)<0 \ \  \text{for} \ \ q\in [q_*,1).
\end{align}  

By continuity of $f$, we have already obtained existence of $\hat{q}\in(\frac{1}{\sqrt{2}}, q_*)$ satisfying \eqref{eq:hat_q-def}. 
It remains to show uniqueness. 
To this end, we calculate 
\begin{align*}
    f'(q)&=(16q^3-10q) \mathrm{K}(q) +(4q^4-5q^2+1)\bigg( \frac{1}{q(1-q^2)}\mathrm{E}(q) -\frac{1}{q}\mathrm{K}(q) \bigg) \\
    &\quad+ (-32q^3+16q)\mathrm{E}(q)+(-8q^4+8q^2-1)\bigg( \frac{1}{q}\mathrm{E}(q) -\frac{1}{q}\mathrm{K}(q) \bigg) \\
    &=(20q^3-13q)\mathrm{K}(q)+20q(-2q^2+1)\mathrm{E}(q).
\end{align*}
By Lemma~\ref{lem:elliptic_2E-K}, $2\mathrm{E}(q)>\mathrm{K}(q)$ holds for $q\in(\frac{1}{2},q_*)$, and hence 
\begin{align}\label{eq:f-diff_estimate}
\begin{split}
    f'(q) < \; (20q^3-13q) \mathrm{K}(q) + 10q(-2q^2+1)\mathrm{K}(q) 
    =-3q\mathrm{K}(q)<0, \quad q\in(\tfrac{1}{\sqrt{2}},q_*).
    \end{split}
\end{align}
Thus $f$ is strictly decreasing on $(\frac{1}{\sqrt{2}}, q_*)$, which yields uniqueness of $\hat{q}\in(\frac{1}{\sqrt{2}}, q_*)$ satisfying \eqref{eq:hat_q-def}.
Combining continuity of $f$ with \eqref{eq:f-negative} and \eqref{eq:f-positive}, we see that $f>0$ on $[\frac{1}{\sqrt{2}}, \hat{q})$ and $f<0$ on $(\hat{q},1)$.
\end{proof}

\begin{proof}[Proof of Lemma~\ref{lem:hat-q-character}]
Throughout this proof we write $m_1:=(\frac{3}{5})^{\frac{1}{2}}$ and $m_2:=(\frac{2}{3})^{\frac{1}{2}}$ for short. 
First we show $m_1<\hat{q}$.
In view of Lemma~\ref{lem:property-f}, it suffices to show $f(m_1)>0$.
To this end, here we adopt ideas from \cite[Proof of Lemma 3.10]{Miura23}, to use the known series expansions of $\mathrm{K}$ and $\mathrm{E}$ (e.g.\ founded in \cite[17.3.11, 17.3.2]{Abramowitz1992}) that are absolutely convergent 
\begin{align*}
    \mathrm{K}(q)&=\int_0^{\frac{\pi}{2}}\frac{1}{\sqrt{1-q^2\sin^2\theta }} \; \mathrm{d}\theta = \frac{\pi}{2}\sum_{n=0}^\infty\bigg(\frac{(2n-1)!!}{(2n)!!}\bigg)^2q^{2n}, \\
    \mathrm{E}(q)&= \int_0^{\frac{\pi}{2}} \sqrt{1-q^2\sin^2\theta} \; \mathrm{d}\theta = \frac{\pi}{2}\sum_{n=0}^\infty\bigg(\frac{(2n-1)!!}{(2n)!!}\bigg)^2\frac{1}{1-2n}q^{2n}. 
\end{align*}
We deduce from these expansions that
\begin{align*}
f(m_1)&=-\frac{14}{25}\mathrm{K}(m_1) + \frac{23}{25}\mathrm{E}(m_1)\\
&=\frac{9\pi}{50}\bigg[1-\sum_{n=1}^\infty A_n \Big(\frac{3}{5}\Big)^n \bigg], \quad \text{where} \quad A_n:=\bigg(\frac{(2n-1)!!}{(2n)!!}\bigg)^2\frac{28n+9}{9(2n-1)}.
\end{align*}
Since $A_n\leq A_1=\frac{37}{36}$ holds by induction, it follows that for all $N\in\mathbf{N}$
\begin{align*}
\sum_{n=1}^\infty A_n \Big(\frac{3}{5}\Big)^n &\leq \sum_{n=1}^N A_n \Big(\frac{3}{5}\Big)^n + \sum_{n=N+1}^\infty\frac{37}{36}\Big(\frac{3}{5}\Big)^n \\
&=\sum_{n=1}^N A_n \Big(\frac{3}{5}\Big)^n + \frac{185}{72}\Big(\frac{3}{5}\Big)^{N+1}=:T_N.
\end{align*}
An explicit computation (of finite operations multiplying integers) shows that 
\[
T_5=\frac{570955201}{614400000}<1, 
\]
and hence we obtain $f(m_1)\geq \frac{9\pi}{50}(1-T_5)>0$. 
\\
\indent
Now we prove $\hat{q}<m_2$ by showing that $f(m_2)<0$. 
Noting that for any $q\in(0,1)$
\begin{align}\label{eq:tech-Elliptic_int}
\begin{split}
    \int_0^{\frac{\pi}{2}}\frac{1-2\sin^2\theta}{\sqrt{1-q^2\sin^2\theta}} \;\mathrm{d}\theta
    &=\int_0^{\frac{\pi}{2}}\frac{\cos{2\theta}}{\sqrt{1-q^2\sin^2\theta}} \;\mathrm{d}\theta \\
    &=\int_0^{\frac{\pi}{4}}\frac{\cos{2\theta}}{\sqrt{1-q^2\sin^2\theta}} \;\mathrm{d}\theta - \int_0^{\frac{\pi}{4}}\frac{\cos{2\theta}}{\sqrt{1-q^2\cos^2\theta}} \;\mathrm{d}\theta \\
    &<0, 
\end{split}
\end{align} 
we obtain
\begin{align*}
\mathrm{E}(m_2) -\frac{2}{3}\mathrm{K}(m_2) = \frac{1}{3}\int_0^{\frac{\pi}{2}}\frac{1-2\sin^2\theta}{\sqrt{1-\frac{2}{3}\sin^2\theta}} \;\mathrm{d}\theta <0. 
\end{align*}
This yields $
f(m_2) = -\frac{5}{9}\mathrm{K}(m_2) + \frac{7}{9}\mathrm{E}(m_2) 
<-\frac{5}{9}\mathrm{K}(m_2) +\frac{14}{27}\mathrm{K}(m_2) <0$. 
\end{proof}

\begin{proof}[Proof of Lemma~\ref{lem:property-e}]
Set $\zeta(q):=(4q^2-3)\mathrm{K}(q)+2\mathrm{E}(q)$ and $\xi(q):=(2q^2-1)(2\mathrm{E}(q)-\mathrm{K}(q))$ (so that $e=\zeta/\xi$). 
Then, in view of \eqref{eq:diff-elliptic-int}, we see that 
\begin{align}
\begin{split} \label{eq:diff-zeta}
\zeta'(q)
&=8q\mathrm{K}(q) + (4q^2-3)\Big( \frac{\mathrm{E}(q)}{q(1-q^2)} - \frac{\mathrm{K}(q)}{q} \Big)+2\Big(\frac{\mathrm{E}(q)}{q} - \frac{\mathrm{K}(q)}{q}\Big)\\
&=\frac{4q^2+1}{q}\mathrm{K}(q)+\frac{2q^2-1}{q(1-q^2)}\mathrm{E}(q).
\end{split}
\end{align}
Similarly it follows that 
\begin{align*}
\xi'(q)
&=4q(2\mathrm{E}(q)-\mathrm{K}(q)) +(2q^2-1) \bigg[2\Big(\frac{\mathrm{E}(q)}{q} - \frac{\mathrm{K}(q)}{q}\Big) + \Big( \frac{\mathrm{E}(q)}{q(1-q^2)} - \frac{\mathrm{K}(q)}{q} \Big) \bigg] \\
&=\frac{-6q^2+1}{q}\mathrm{K}(q)+\frac{-12q^4+12q^2-1}{q(1-q^2)}\mathrm{E}(q).
\end{align*}
A straightforward computation shows that 
\begin{align*}
\zeta'(q)\xi(q) 
&=(2q^2-1) \bigg( -\frac{4q^2+1}{q}\mathrm{K}(q)^2 +2\frac{2q^2-1}{q(1-q^2)}\mathrm{E}(q)^2 \\
& \hspace{133pt} +\frac{-8q^4+4q^2+3}{q(1-q^2)}\mathrm{K}(q)\mathrm{E}(q) \bigg), \\
\zeta(q)\xi'(q) 
&= \frac{(4q^2-3)(-6q^2+1)}{q}\mathrm{K}(q)^2 +2\frac{-12q^4 +12q^2-1}{q(1-q^2)}\mathrm{E}(q)^2 \\
& \hspace{129pt} + \frac{-48q^6+96q^4 -54q^2+5}{q(1-q^2)}\mathrm{K}(q)\mathrm{E}(q).
\end{align*}
Therefore, we obtain 
\begin{align*}
&\zeta'(q)\xi(q) - \zeta(q)\xi'(q) \\
=&\ \frac{16q^4-20q^2+4}{q}\mathrm{K}(q)^2 +\frac{32q^4-32q^2+4}{q(1-q^2)}\mathrm{E}(q)^2 \\
&\quad\quad\quad + \frac{32q^6-80q^4+56q^2-8}{q(1-q^2)}\mathrm{K}(q)\mathrm{E}(q) \\
=&\ 4\Big((4q^4-5q^2+1) \mathrm{K}(q) +(-8q^4+8q^2-1)\mathrm{E}(q) \Big)\Big(\frac{1}{q}\mathrm{K}(q) - \frac{1}{q(1-q^2)}\mathrm{E}(q) \Big) \\
=&\ -4f(q)\mathrm{K}'(q),
\end{align*}
where we used \eqref{eq:def-f} and \eqref{eq:diff-elliptic-int} in the last equality. 
For $q\in(\frac{1}{\sqrt{2}}, \hat{q})$, we obtain \eqref{eq:diff-e} since $e'=(\zeta'\xi-\zeta \xi')/\xi^2$.
On the other hand, for $q\in(\hat{q}, 1)$, noting $e'=-(\zeta'\xi-\zeta \xi')/\xi^2$ and the fact that $2\mathrm{E}(q)-\mathrm{K}(q)<0$, we see that \eqref{eq:diff-e} holds as well. 

In addition, we can deduce \eqref{eq:monotonicity-e} from \eqref{eq:diff-e} combined with Lemma~\ref{lem:property-f}.
\end{proof}

\begin{proof}[Proof of Lemma~\ref{lem:property-h}]
Let $\zeta(q):=(4q^2-3)\mathrm{K}(q)+2\mathrm{E}(q)$ as in the proof of Lemma~\ref{lem:property-e}. 
Then, due to the fact that $h(q)=\frac{\zeta(q)}{\sqrt{2q^2-1}}$ and by \eqref{eq:diff-zeta} we have 
\begin{align*}
    h'(q)&= \frac{1}{(2q^2-1)^{\frac{3}{2}}} \big(-2q\zeta(q)+(2q^2-1)\zeta'(q) \big) \\
    &= \frac{1}{(2q^2-1)^{\frac{3}{2}}} \bigg(-2q\Big((4q^2-3)\mathrm{K}(q)+2\mathrm{E}(q)\Big) \\
    &\quad\quad\quad\quad\quad\quad\quad +(2q^2-1)\Big(\frac{4q^2+1}{q}\mathrm{K}(q)+\frac{2q^2-1}{q(1-q^2)}\mathrm{E}(q) \Big) \bigg) \\
    &= \frac{1}{(2q^2-1)^{\frac{3}{2}}} \bigg( \Big(4q-\frac{1}{q}\Big)\mathrm{K}(q) + \frac{8q^4-8q^2+1}{q(1-q^2)} \mathrm{E}(q) \bigg) \\
    &=- \frac{1}{(2q^2-1)^{\frac{3}{2}}q(1-q^2)}f(q), 
\end{align*}
which completes the proof. 
\end{proof}

\bibliographystyle{abbrv}
\bibliography{ref_Mueller-Yoshizawa}

\end{document}